\documentclass[reqno]{amsart}

\usepackage{amsmath}
\usepackage{amsfonts}
\usepackage{amsthm}
\usepackage{amssymb}
\usepackage{graphicx}
\usepackage{graphics}
\usepackage{bm}
\usepackage{dsfont}
\usepackage[all]{xy}
\usepackage{color}
\usepackage[font=footnotesize]{caption}
\numberwithin{equation}{section}
\newtheoremstyle{personal}%
{12pt}%      Space above
{12pt}%      Space below
{\slshape}%         Body font
{}%         Indent amount
{\bfseries}% Theorem head font
{.}%        Punctuation after theorem head
{.5em}%     Space after theorem head
{}%         Theorem head spec (can be left empty, meaning "normal")
\theoremstyle{personal}%
\newtheorem{thm}{Theorem}[section]
\newtheorem{cor}[thm]{Corollary}
\newtheorem{lem}[thm]{Lemma}
\newtheorem{prop}[thm]{Proposition}
\theoremstyle{definition}

\newtheorem{rem}[thm]{Remark}
\newtheorem{exm}[thm]{Example}

\definecolor{gray}{gray}{0.4}

\newcommand{\length}{\mathrm{length}}
\newcommand{\N}{\mathds{N}}
\newcommand{\Z}{\mathds{Z}}
\newcommand{\R}{\mathds{R}}
\newcommand{\Q}{\mathds{Q}}

\newcommand{\diff}{\mathrm{d}}
\newcommand{\dist}{d}
\newcommand{\crit}{\mathrm{crit}}
\newcommand{\supp}{\mathrm{supp}}
\newcommand{\hess}{\diff^2}
\newcommand{\Tan}{\mathrm{T}}
\newcommand{\cu}{c_{\mathrm{u}}}
\newcommand{\cw}{c_{\mathrm{w}}}
\newcommand{\Hom}{\mathrm{H}}

\newcommand{\SSS}{\mathcal{S}}

\newcommand{\ggamma}{\bm{\gamma}}
\newcommand{\zzeta}{\bm{\zeta}}
\newcommand{\ttheta}{\bm{\theta}}
\newcommand{\eeta}{\bm{\eta}}

\newcommand{\Meas}{\mathfrak{M}}
\newcommand{\Mather}{\mathcal{M}}
\newcommand{\Graph}{\mathcal{G}}

\newcommand{\inj}{{\mathrm{inj}}}

\newcommand{\Mult}{\mathcal{M}}

\newcommand{\ELMult}{\mathcal{D}}
\newcommand{\W}{W^{1,2}}
\newcommand{\UU}{\mathcal{U}}
\newcommand{\VV}{\mathcal{V}}
\newcommand{\WW}{\mathcal{W}}

\newcommand{\PP}{\mathcal{P}}

\DeclareMathOperator*{\toup}{\longrightarrow}

\newcommand{\tb}{\mathcal{B}}

\begin{document}

\title{Minimal boundaries in Tonelli Lagrangian systems}

\author[L. Asselle]{Luca Asselle}
\address{Luca Asselle\newline\indent
Justus Liebig Universit\"at Gie\ss en, Mathematisches Institut\newline\indent Arndtstra\ss e 2, 35392 Gie\ss en, Germany}
\email{luca.asselle@ruhr-uni-bochum.de}

\author[G. Benedetti]{Gabriele Benedetti}
\address{Gabriele Benedetti\newline\indent 
Universit\"at Heidelberg, Mathematisches Institut\newline\indent Im Neuenheimer Feld 205, 69120 Heidelberg, Germany}
\email{gbenedetti@mathi.uni-heidelberg.de}

\author[M. Mazzucchelli]{Marco Mazzucchelli}
\address{Marco Mazzucchelli\newline\indent 
CNRS, \'Ecole Normale Sup\'erieure de Lyon, UMPA\newline\indent  
46 all\'ee d'Italie, 69364 Lyon Cedex 07, France}
\email{marco.mazzucchelli@ens-lyon.fr}

\date{April 25, 2018}

\begin{abstract}
We prove several new results concerning action minimizing periodic orbits of Tonelli Lagrangian systems on an oriented closed surface $M$. More specifically, we show that for every energy larger than the maximal energy of a constant orbit and smaller than or equal to the Ma\~n\'e critical value of the universal abelian cover, the Lagrangian system admits a minimal boundary, i.e.\ a global minimizer of the Lagrangian action on the space of smooth boundaries of open sets of $M$. We also extend the celebrated graph theorem of Mather in this context: in the tangent bundle $\mathrm{T} M$, the union of the supports of all lifted minimal boundaries with a given energy projects injectively to the base $M$. Finally, we prove the existence of action minimizing simple periodic orbits on energies just above the Ma\~n\'e critical value of the universal abelian cover. This provides in particular a class of non-reversible Finsler metrics on the 2-sphere possessing infinitely many closed geodesics.
\tableofcontents
\end{abstract}

\maketitle

\vspace{-5mm}

\section{Introduction}
In low dimensional convex Hamiltonian dynamics, an important role is played by periodic orbits that locally minimize the action. An instance of this can be seen in the remarkable result of Bangert \cite{Bangert:1980ho} showing that the existence of a locally length minimizing closed geodesic -- a so-called ``waist'' -- on a 2-sphere forces the existence of infinitely many more closed geodesics. A Riemannian or Finsler 2-sphere does not necessarily have waists, and indeed there are Finsler spheres with only finitely many closed geodesics \cite{Katok:1973mw, Ziller:1983rw}. However, Taimanov \cite{Taimanov:1991el, Taimanov:1992sm} showed that magnetic geodesic flows on closed surfaces always have waists on sufficiently small energy levels. This result has been given an independent proof and put in the context of Aubry-Mather theory by Contreras, Macarini, and Paternain \cite{Contreras:2004lv}. Recently, its validity has been extended to the class of Tonelli Lagrangian system, the largest class Aubry-Mather theory deals with, by the first and third author \cite{Asselle:2016qv}. In the current paper, we improve these results and derive new, rather unexpected, applications.

%%%%%%%%%%%%%%%%%%%%

\subsection{The setting}
We recall that a Tonelli Lagrangian $L:\Tan M\to\R$, where $M$ is a closed manifold, is a smooth function whose restriction to the fibers of the tangent bundle $\Tan M$ is superlinear with positive definite Hessian. Its Euler-Lagrange flow $\phi_L^t:\Tan M\to\Tan M$ is a second order flow on $M$: a smooth curve $\gamma:\R\to M$ is a Lagrangian orbit if and only if it satisfies the Euler-Lagrange equation $\tfrac{\diff}{\diff t}\partial_vL(\gamma,\dot\gamma)-\partial_qL(\gamma,\dot\gamma)=0$; the corresponding flow line is given by $\phi_L^t(\gamma(0),\dot\gamma(0))=(\gamma(t),\dot\gamma(t))$. The energy 
\[E:\Tan M\to\R,\qquad E(q,v)=\partial_vL(q,v)v-L(q,v)\] is a first integral of the dynamics, meaning that $E\circ\phi_L^t=E$ for all $t\in\R$. 

In this paper, we focus on periodic orbits, that is, those orbits of the form $\gamma:\R/p\Z\to M$ for some period $p>0$. 
The Lagrangian action with energy $e\in\R$ of an absolutely continuous periodic curve $\gamma:\R/p\Z\to M$ is the quantity
\begin{align*}
 \SSS_e(\gamma) = \int_0^p L(\gamma(t),\dot\gamma(t))\,\diff t + p\,e \in \R\cup\{+\infty\}.
\end{align*}
The functional $\SSS_e$, whose domain is the space of absolutely continuous periodic curves with arbitrary period, is known as the free-period action. A version of the least action principle implies that the periodic orbits $\gamma$ of the Lagrangian system of $L$ with energy $e:=E(\gamma,\dot\gamma)$ are precisely the critical points of $\SSS_e$. The simplest example of critical points are the minimizers: we will call \textbf{Tonelli waist} with energy $e$ a local minimizer of $\SSS_e$. We refer the reader to \cite{Contreras:2006yo, Abbondandolo:2013is} for the background on the variational properties of the free-period action functional.

The qualitative properties of the Lagrangian dynamics depend on the energy level that one considers. A first, significant, energy level is $e_0(L):= \max_{q\in M} E(q,0)$.
This is the lowest level such that, for all $e$ above it, there are orbits with energy $e$ going through any given point of $M$. Another remarkable energy level is the Ma\~n\'e critical value $c(L)$, which is defined as the minimal energy $e\in\R$ such that the free-period action functional $\SSS_e$ is non-negative. Given a covering space $M'$ of $M$, one can consider the Ma\~n\'e critical value of the lift of the Lagrangian $L$ to $\Tan M'$. When $M'$ is equal to the universal abelian cover or to the universal cover, the respective Ma\~n\'e critical values are usually denoted by $c_0(L)$ and $\cu(L)$. The energy values that we have introduced so far are ordered as
\begin{align*}
e_0(L)\leq \cu(L)\leq c_0(L) \leq c(L).
\end{align*}
Throughout this paper, we will always assume that $e_0(L)<c_0(L)$. This is a mild assumption, and indeed the strict inequality $e_0(L)<\cu(L)$ is verified on a $C^0$-open and $C^1$-dense subspace of the space of Tonelli Lagrangians (Proposition~\ref{p:T'}).

On any energy level $e>c_0(L)$, the Lagrangian dynamics is of Finsler type, that is, the Euler-Lagrange flow on $E^{-1}(e)$ is orbitally equivalent to a Finsler geodesic flow on the unit tangent bundle of $M$, see \cite[Cor.~2]{Contreras:1998lr}. On energy levels $e\leq c_0(L)$, the Lagrangian  dynamics is in general different from a Finsler one, and particularly when $e\leq \cu(L)$ its study poses several issues due to the potential lack of good variational properties of $\SSS_e$, see \cite{Contreras:2006yo, Abbondandolo:2013is, Abbondandolo:2015lt, Abbondandolo:2014rb}. 

Building on the above mentioned Taimanov's work, Contreras, Macarini, and Paternain \cite{Contreras:2004lv} showed that, when $M$ is a closed surface and $L$ is magnetic (that is, of the form $L(q,v)=\tfrac12g_q(v,v)+\theta_q(v)$ for some Riemannian metric $g$ and 1-form $\theta$), for any $e\in(e_0(L),\cu(L))$ there exists a (not necessarily simple) Tonelli waist $\gamma$ with energy $e$ and  negative action $\SSS_e(\gamma)<0$.\footnote{The above mentioned result in \cite{Contreras:2004lv} is actually claimed for all  $e\in(e_0(L),c_0(L))$, but the proof contains a mistake, which makes the argument correct only for $e\in(e_0(L),\cu(L))$. The mistake is in the technical lemma \cite[Lemma~3.3]{Contreras:2004lv}: the claim that the multicurve $\tau$ is null-homologous is not always true. This is only an issue for energies $e\geq \cu(L)$. Indeed, when $e<\cu(L)$, the curve $\tau$ can be chosen to be the lift of a contractible curve, and in particular it is null-homologous. Nevertheless, the impact of the mistake on the main result of \cite{Contreras:2004lv} is essentially at level $\cu(L)$. Indeed,  the strict inequality $\cu(L)<c_0(L)$ can only hold for surfaces of genus at least 2, which have non-abelian fundamental group, and for each energy $e>\cu(L)$ the existence of a Tonelli waist becomes elementary: $\SSS_e$ satisfies the Palais-Smale condition, and one can find global minimizers of $\SSS_e$ in every connected component of non-contractible periodic curves. Our Theorem~\ref{t:below_Mane} in particular recovers the full result of \cite{Contreras:2004lv} for all $e\in(e_0(L),c_0(L))$.} 
In the recent paper \cite{Asselle:2016qv}, the first and third authors extended the validity of this result to general Tonelli Lagrangians on closed surfaces. In this paper, we further strengthen the result: on any energy level $e\in(e_0(L),c_0(L))$, we will show the existence of simple Tonelli waists, that is, Tonelli waists that have no self-intersections. Actually, as in \cite{Taimanov:1991el, Taimanov:1992sm, Contreras:2004lv}, we will actually show the existence of particular multicurves of periodic orbits.

Throughout this paper, by a multicurve we mean a collection $\ggamma=(\gamma_1,...,\gamma_m)$ of finitely many absolutely continuous periodic curves $\gamma_i:\R/p_i\Z\to M$. Each $\gamma_i$ is called a component of $\ggamma$. When the $\gamma_i$'s are topologically embedded and have pairwise disjoint image, the multicurve $\ggamma$ is said to be embedded and can be seen as an oriented 1-dimensional topological submanifold of $M$ with connected components $\gamma_1,...,\gamma_m$.   
On a closed oriented surface $M$, an embedded multicurve $\ggamma$ is a topological boundary when it is the piecewise smooth (with finitely many singular points) oriented boundary of an open subset $\Sigma\subset M$. 
We denote by $\tb$ the space of topological boundaries of $M$. 
The free-period action functional admits a natural extension, that we will still denote by $\SSS_e$, on the space of multicurves in $M$: given any such multicurve $\ggamma=(\gamma_1,...,\gamma_m)$, we set
$\SSS_e(\ggamma):=\SSS_e(\gamma_1)+...+\SSS_e(\gamma_m)$. 
We say that a topological boundary $\ggamma=(\gamma_1,...,\gamma_m)$ is a \textbf{minimal boundary} with energy $e$ for the Tonelli Lagrangian $L:\Tan M\to\R$ when $\SSS_e(\ggamma) = \inf_{\tb}\SSS_e$. If $e>e_0(L)$, the components of minimal boundaries with energy $e$ are simple Tonelli waists with energy $e$ (Lemma~\ref{l:minimal_topological_boundaries}).

%%%%%%%%%%%%%%%%%%%%%

\subsection{Results on subcritical energies}
The first result of our paper implies, in particular, the existence of simple Tonelli waists in the energy range $(e_0(L),c_0(L)]$.

\begin{thm}\label{t:below_Mane}
Let $M$ be an oriented closed surface, and $L:\Tan M\to\R$ a Tonelli Lagrangian with $e_0(L)<c_0(L)$. For each energy value $e\in(e_0(L),c_0(L)]$, there exists a minimal boundary $\ggamma$ with energy $e$ for $L$ with action  $\SSS_{e}(\ggamma)<0$ if $e<c_0(L)$, or $\SSS_{e}(\ggamma)=0$ if $e=c_0(L)$. 
\end{thm}

Let $\Meas_0(L)$ be the set of Borel probability measures $\mu$ on the tangent bundle $\Tan M$ that are invariant by the Euler-Lagrange flow $\phi_L^t$ and have rotation vector 
$\rho(\mu)=0$. We denote by $\Meas_{\min}(L)$ the set of minimal measures with zero rotation vector, that is, those $\mu\in\Meas_0(L)$ such that $\SSS(\mu)=\inf_{\Meas_0(L)}\SSS$, where $\SSS$ is the action functional on probability measures
\begin{align*}
 \SSS(\mu)=\int_{\Tan M} L\,\diff\mu.
\end{align*}
One of the earliest results of Aubry-Mather theory asserts that $\inf_{\Meas_0(L)}\SSS = -c_0(L)$, and that any minimal measure with zero rotation vector has support in the energy level $E^{-1}(c_0(L))$. Moreover, Mather's graph theorem asserts that the base projection $\pi:\Tan M\to M$ restricts to an injective map on the Mather set
\begin{align*}
\Mather_0(L):=\bigcup_{\mu\in\Meas_{\min}(L)}\!\!\! \mathrm{supp}(\mu),
\end{align*}
see \cite{Mather:1991xd, Contreras:1999fm, Fathi:2009fu}.

The subcritical energy levels are unaccessible by Aubry-Mather theory, but nevertheless the minimal boundaries provide invariant sets by the Euler-Lagrange flow analogous to the Mather set. Indeed, for each $e\in(e_0(L),c_0(L)]$, we define
\begin{align*}
\Graph_e(L)\subset E^{-1}(e) 
\end{align*}
to be the union, over all components $\gamma$ of all minimal boundaries with energy $e$, of all points of the form $(\gamma(t),\dot\gamma(t))$. If $e_0(L)<c_0(L)$, the set $\Graph_{c_0(L)}(L)$ turns out to be precisely the Mather set $\Mather_0(L)$ (Proposition~\ref{p:Mather}). Therefore, our next theorem can be seen as an extension of Mather's graph theorem to subcritical energy levels on closed oriented surfaces.

%In Corollary~\ref{c:irreducible_minimal}, we see that $\Mather_0(L)$ always contains a topological boundary with at most $g+1$ components, $g$ being the genus of $M$. 
%This result turns out to be optimal (c.f. discussion at the end of \S \ref{s:graph}). In the next result, we show that $\mathcal G_e$ still enjoys the graph property for $e<c_0(L)$.

\begin{thm}[Subcritical graph theorem] \label{t:graph}
Let $M$ be an oriented closed surface, and $L:\Tan M\to\R$ a Tonelli Lagrangian with $e_0(L)<c_0(L)$.
For each energy value $e\in (e_0(L),c_0(L)]$, the restriction $\pi|_{\Graph_e(L)}:\Graph_e(L) \rightarrow M$ is injective. 
\end{thm}

As an application of the subcritical graph theorem and of Theorem~\ref{t:below_Mane}, we prove the existence of simple Tonelli waists on subcritical energy levels of non-orientable closed surfaces, see Theorem \ref{t:non_orientable}.

\subsection{Results on supercritical energies}
Let $M$ be a closed oriented surface, and $L:\Tan M\to\R$ a Tonelli Lagrangian. It is well-known that, if $e>c_0(L)$ and $M$ has positive genus,  there exist infinitely many Tonelli waists with energy $e$, and at least a simple one. On the other hand, there are never minimal boundaries for $e>c_0(L)$, since $\inf_{\mathcal B}\SSS_e=0$ and the infimum is not attained. Our next theorem shows that there are at least locally minimal boundaries.

\begin{thm}
\label{t:just_above}
Let $M$ be an oriented closed surface, and $L:\Tan M\to\R$ a Tonelli Lagrangian such that $e_0(L)<c_0(L)$. There exists $\cw(L)>c_0(L)$ and, for each $e\in(c_0(L),\cw(L))$, a topological boundary $\ggamma=(\gamma_1,...,\gamma_m)$ whose components are (simple) Tonelli waists for $L$ with energy $e$.
\end{thm}
 
The assumption $e_0(L)<c_0(L)$ in Theorem~\ref{t:just_above} cannot be dropped. Indeed, if $g$ denotes the round Riemannian metric on $S^2$, 
the Tonelli Lagrangian $L:\Tan S^2\to\R$, $L(q,v)= \frac 12 g_q(v,v)$  satisfies $e_0(L)=c_0(L)=c(L)=0$ and does not have Tonelli waists on any energy level $e>0$. 
Theorem~\ref{t:just_above} becomes particularly significant precisely when $M=S^2$, as there are examples due to Katok \cite{Katok:1973mw, Ziller:1983rw} of Tonelli Lagrangians $L:\Tan S^2\to\R$ with $e_0(L)<c(L)$ and such that, on some energy level $e>c(L)$, there are only finitely many periodic orbits and no Tonelli waists at all. 

The waist theorem for Tonelli Lagrangians on surfaces, which was stated without proof in \cite[Corollary~2.7]{Abbondandolo:2014rb}, implies in particular that the presence of a contractible Tonelli waist on energy levels $e>\cu(L)$ forces the existence of infinitely many other contractible periodic orbits with energy $e$. This, together with Theorem~\ref{t:just_above}, implies a new multiplicity result for Tonelli periodic orbits on $S^2$ at energies just above the $c(L)$. 

\begin{thm}\label{t:multiplicity_just_above}
Let $L:\Tan S^2\to\R$ be a Tonelli Lagrangian such that $e_0(L)<c(L)$. For all energy levels $e\in(c(L),\cw(L))$, the Lagrangian system of $L$ admits infinitely many periodic orbits with energy $e$.
\end{thm}

We will show in Section~\ref{s:Finsler} that Theorem~\ref{t:multiplicity_just_above} provides a class of non-reversible Finsler metrics (of Randers type, as the Katok's one \cite{Katok:1973mw, Ziller:1983rw}) on $S^2$ having infinitely many closed geodesics.

%%%%%%%%%%%%%%%%%%%%%

\subsection{Organization of the paper}
In Section~\ref{s:orientable}, after some preparatory lemmas on embedded multicurves on surfaces, we prove the four main theorems stated in the introduction. In Section~\ref{s:non_orientable}, we prove Theorem~\ref{t:non_orientable} on the existence of simple Tonelli waists on subcritical energy levels in non-orientable closed surfaces. In Section~\ref{s:suff_cond} we provide a sufficient condition for the inequality $e_0(L)<\cu(L)$, and show in particular that it is verified for a $C^1$-generic Tonelli Lagrangian. 
In Section~\ref{s:Finsler} we study the implications of Theorem~\ref{t:multiplicity_just_above} to Finsler dynamics on $S^2$. In the Appendix, we prove a remark concerning Tonelli waists in the $\W$ functional setting of the free-period action functional.

%%%%%%%%%%%%%%%%%%%%%

\subsection{Acknowledgements}
We are grateful to Gonzalo Contreras, Leonardo Macarini, and Gabriel Paternain for a helpful discussion concerning their paper \cite{Contreras:2004lv}. Luca Asselle is partially supported by the DFG-grants AB 360/2-1 ``Periodic orbits of conservative systems below the Ma\~n\'e critical energy value'' and AS 546/1-1 ``Morse theoretical methods in Hamiltonian dynamics''. Marco Mazzucchelli is partially supported by the ANR COSPIN (ANR-13-JS01-0008-01).

%%%%%%%%%%%%%%%%%%%%%
%%%%%%%%%%%%%%%%%%%%%
%%%%%%%%%%%%%%%%%%%%%

\section{The orientable case}
\label{s:orientable}

%%%%%%%%%%%%%%%%%%%%%

\subsection{Homological versus topological boundaries}\label{s:hom_vs_top}

Let $M$ be a closed oriented surface. A multicurve $\ggamma$ in $M$ defines a homology class $[\ggamma]\in\Hom_1(M;\Z)$, and is called a homological boundary when $[\ggamma]=0$. Clearly, a topological boundary in $M$ is also a homological boundary. Conversely, we have the following statement. Given two multicurves $\ggamma=(\gamma_1,...,\gamma_m)$ and $\zzeta=(\zeta_1,...,\zeta_n)$ in $M$, their union is the multicurve $\ggamma\cup\zzeta=(\gamma_1,...,\gamma_m,\zeta_1,...,\zeta_n)$. 

\begin{lem}\label{l:boundaries}
An embedded multicurve $\ggamma$ is a homological boundary if and only if it can be written as a disjoint union of topological boundaries. 
\end{lem}

\begin{proof}
We only have to prove the ``only if'' direction of the statement. Therefore, let us assume that the embedded multicurve $\ggamma=(\gamma_1,...,\gamma_m)$ is a homological boundary. Let $\Sigma_1,...,\Sigma_k$ be the connected components of $M\setminus\ggamma$. The oriented boundary of $\Sigma_i$ is an (a priori not necessarily disjoint) union $\partial\Sigma_i=\partial_+\Sigma_i\cup\partial_-\Sigma_i$, where the orientations of $\partial\Sigma_i$ and $\ggamma$ coincide on $\partial_+\Sigma_i$ and differ on $\partial_-\Sigma_i$. We introduce a suitable triangulation of $M$ such that the periodic curves $\gamma_i$ and the closures of the regions $\Sigma_j$ define simplicial chains that we still denote by $\gamma_i$ and $\Sigma_j$ with a common abuse of notation. Since the multicurve $\ggamma$ is a homological boundary, we can find a 2-chain $\Pi$ such that $\ggamma=\partial\Pi$.

We claim that there exist $n_1,...,n_k\in\Z$ such that $\Pi=n_1\Sigma_1+...+n_k\Sigma_k$. Indeed, $\Pi$ is an integer linear combination of the 2-simplexes of the triangulation; if $\Delta$ and $\Delta'$ are 2-simplexes contained in the closure of the same $\Sigma_i$, then their coefficients must be the same; this is clear if $\Delta$ and $\Delta'$ are adjacent and, in general, follows from the connectedness of $\Sigma_i$. We conclude that
\begin{align}\label{e:chain_equality}
\gamma_1+\ldots+\gamma_m
=
n_1(\partial_+\Sigma_1 + \partial_-\Sigma_1) +\ldots+n_k (\partial_+\Sigma_k + \partial_-\Sigma_k).
\end{align}
We define the functions 
$
\iota_{\pm}:\{1,...,m\}\to\{1,...,k\} 
$
such that, for each $i=1,...,m$, the oriented loop $\gamma_i$ belongs to $\partial_+\Sigma_{\iota_+(i)}$ and, with reverse orientation, to $\partial_-\Sigma_{\iota_-(i)}$. Equation~\eqref{e:chain_equality} implies that 
\begin{align}\label{e:n_i}
n_{\iota_+(i)} = n_{\iota_-(i)} + 1,\qquad\forall i=1,...,m.
\end{align}

We set $J_1:=\{1\}$ and proceed iteratively, for increasing values of $h$ starting at $h=1$, as follows: if $J_h\neq\varnothing$, we set
\begin{align*}
J_{h+1}
:=
\bigcup_{j\in J_h}
\iota_+(\iota_-^{-1}(j)).
\end{align*}
Notice that, if $j\in J_h$ and $j'\in \iota_+(\iota_-^{-1}(j))$, Equation~\eqref{e:n_i} implies that $n_{j'}=n_j+1$. Therefore $J_{h+1}\cap \big( J_1\cup...\cup J_h\big) =\varnothing$, which implies that the iterative procedure will eventually stop, giving $J_{h+1}=\varnothing$ for some integer $h$. This means that the oriented boundary $\ggamma'$ of the compact subset
\begin{align*}
\Sigma'
:=
\bigcup_{j\in J_1\cup...\cup J_h} \!\!\!\! \overline{\Sigma}_j.
\end{align*}
is a union of some connected components of $\ggamma$. Now we repeat the same procedure with the homological boundary $\ggamma_1:= \ggamma \setminus \ggamma'$; 
the process clearly stops after finitely many steps, thus giving the desired decomposition of $\ggamma$ as the disjoint union of topological boundaries. 
\end{proof}

We call an embedded homological boundary  \textbf{irreducible} if it cannot be decomposed as a disjoint union of two non-empty homological boundaries. Lemma~\ref{l:boundaries} implies that irreducible homological boundaries are indeed topological boundaries. We recall that a non-zero homology class $h\in\Hom_1(M;\Z)$ is called primitive when it is not of the form $h=nk$, with $n>1$ and $k\in\Hom_1(M;\Z)$.

\begin{lem}\label{l:irreducible}
Let $\ggamma=(\gamma_1,...,\gamma_m)$ be an irreducible topological boundary on an oriented closed surface $M$ of genus $g$. Then $m\leq g+1$ and, if $m>1$, all the classes $[\gamma_i]\in \Hom_1(M;\Z)$ are non-zero, primitive and pairwise distinct.
\end{lem}
\begin{proof}
If $[\gamma_{i_0}]=0$ for some $i_0$, then $\gamma_{i_0}$ is itself a topological boundary. Therefore, if we assume that $m>1$, all the classes $[\gamma_i]\in \Hom_1(M;\Z)$ are non-zero. Since every component $\gamma_i$ is a simple curve, its homology class $[\gamma_i]$ is primitive, see \cite[Proposition~1.4]{Farb:2012rc}. 

Let us assume by contradiction that $[\gamma_i]=[\gamma_j]$ for some $i\neq j$. We denote $\overline{\gamma_j}$ the curve $\gamma_j$ with opposite orientation. The multicurve $\ggamma'=(\gamma_i,\overline{\gamma_j})$ is a homological boundary, and since both $[\gamma_i]$ and $[\overline{\gamma_j}]=-[\gamma_j]$ are non-zero in $\Hom_1(M;\Z)$, Lemma~\ref{l:boundaries} implies that $\ggamma'$ is a topological boundary. Let $\Sigma'\subset M\setminus\ggamma'$ be the connected component whose oriented boundary is $\ggamma'$. The oriented boundary of the open subset $\Sigma'':=\Sigma\cup\gamma_j\cup\Sigma'\subset M$ is a topological boundary $\ggamma''$. Notice that $\gamma_j$ is not a component of $\ggamma''$. Therefore $\ggamma''$ is strictly contained in $\ggamma$, which contradicts the irreducibility of $\ggamma$.

Finally, assume by contradiction that $m>g+1$. Let $V\subset \Hom_1(M;\Q)$ be the vector space over the rational numbers generated by the homology classes $[\gamma_1],...,[\gamma_m]$. We recall that $\Hom_1(M;\Q)$ can be equipped with a symplectic bilinear form $\omega$ given by $\omega(h,k)=h\cap k$, where $\cap$ denotes the homology intersection product. Since the curves $\gamma_i$ are pairwise disjoint, $V$ is an isotropic subspace of $\Hom_1(M;\mathbb Q)$, and in particular $\dim V\leq \tfrac12\dim\Hom_1(M;\Q)= g$. Since $\ggamma$ is null-homologous, the origin in $V$ is in the convex hull of the set $\{[\gamma_1],...,[\gamma_m]\}$. By Caratheodory's Theorem, there are $g+1$ non-negative integers $n_1,...,n_{g+1}$ and $g+1$ components of $\ggamma$, say $\gamma_1,...,\gamma_{g+1}$, such that
\[
0=n_1[\gamma_1]+\ldots+n_{g+1}[\gamma_{g+1}].
\]
Let $\ggamma'$ be an auxiliary embedded homological boundary $\ggamma'$ with $n_i$ components in the class $[\gamma_i]$, for $i=1,...,g+1$. By Lemma~\ref{l:boundaries}, we can write $\ggamma'$ as a disjoint union $\ggamma'=\ggamma''\cup\ggamma'''$ for some irreducible topological boundary $\ggamma''=(\gamma''_1,...,\gamma''_{m'})$. We already proved that the components of irreducible topological boundaries have pairwise distinct homology classes. Therefore, there exist pairwise distinct indices $i_1,...,i_{m'}\in\{1,...,g+1\}$ such that $[\gamma''_j]=[\gamma_{i_j}]$ for all $j=1,...,m'$. In particular,  $m'\leq g+1$ and $[\gamma_{i_1}]+...+[\gamma_{i_{m'}}]=0$. This contradicts the irreducibility of $\ggamma$.
\end{proof}

%%%%%%%%%%%%%%%%%%%%%%

\subsection{Subcritical minimal boundaries}\label{s:below_Mane}

Let $M$ be a closed oriented surface. For a number $m\in\N$, we denote by $\tb(m)\subset\tb$ the subspace of topological boundaries in $M$ with $m$ components. We endow $\tb(m)$ with the absolutely continuous topology, and we consider its closure $\overline{\tb(m)}$, which is the space of ``pinched'' topological boundaries. Notice that the components of the multicurves in $\overline{\tb(m)}$ are allowed to intersect one another, but only with tangencies.

Let $L:\Tan M\to\R$ be a Tonelli Lagrangian. The notion of injectivity radius from Riemannian geometry generalizes to Tonelli systems as follows (see \cite[Subsection~2.1]{Asselle:2016qv} for the proofs). Let $g$ be an auxiliary Riemannian metric on $M$, which induces a  distance $d:M\times M\to[0,\infty)$. For each energy value $e>e_0(L)$ there exist $\tau_\inj>0$ and $\rho_\inj>0$ such that, for each $q_0,q_1\in M$ with $d(q_0,q_1)\leq\rho_\inj$ there exists $\tau\in[0,\tau_\inj]$ and a smooth curve $\gamma:[0,\tau]\to M$ that is a unique local free-time action minimizer with endpoints $\gamma(0)=q_0$ and $\gamma(\tau)=q_1$. This means that, if $B\subset M$ denotes the closed Riemannian ball of radius $\rho_\inj$ centered at $q_0$, for any other absolutely continuous curve $\zeta:[0,\sigma]\to B$ such that $\zeta(0)=q_0$ and $\zeta(\sigma)=q_1$, we have
\begin{align*}
\int_0^\tau L(\gamma(t),\dot\gamma(t))\,\diff t + \tau e
<
\int_0^\sigma L(\zeta(t),\dot\zeta(t))\,\diff t + \sigma e.
\end{align*}
For $m\in\N$, we denote by $\ELMult_e(m)$ the space of absolutely continuous multicurves with $m$ components $\ggamma=(\gamma_1,...,\gamma_m)$ such that each $\gamma_i$ is the concatenation of a finite number of unique local  free-time action minimizers. We set
\begin{align*}
 \Mult_e & := \bigcup_{m\in\N}  \overline{\tb(m)} \cap \ELMult_e(m).
\end{align*}
In \cite[Lemmas~2.7-2.8]{Asselle:2016qv}, the first and third authors investigated the properties of the minimizers of the free-period action functional $\SSS_e$ over the spaces $\overline{\tb(m)} \cap \ELMult_e(m)$, and in particular proved the following.
\begin{lem}\label{l:minima_are_embedded}
For every $e>e_0(L)$, any multicurve $\ggamma\in\Mult_e$ satisfying $\SSS_e(\ggamma)=\inf\SSS_e|_{\Mult_e}$ is embedded, and its components are Tonelli waists with energy $e$.
\hfill\qed
\end{lem}
The proof of \cite[Theorem~1.1]{Asselle:2016qv} actually implies the following statement.

\begin{lem}\label{l:negative_implies_waist}
If, for some $e>e_0(L)$, $\SSS_e$ attains negative values on the space $\Mult_e$, then there exists a multicurve $\ggamma\in\Mult_e$ such that $\SSS_e(\ggamma)=\inf\SSS_e|_{\Mult_e}$. 
\hfill\qed
\end{lem}

Lemma~\ref{l:negative_implies_waist} is always applicable on subcritical energy levels, according to the following statement.

\begin{lem}\label{l:minima_are_negative}
For each $e\in(e_0(L),c_0(L))$,  $\SSS_e$ attains negative values on $\Mult_e$.
\end{lem}

\begin{proof}
Since $e\in(e_0(L),c_0(L))$, there exists an absolutely continuous and null-homologous curve $\zeta:\R/s\Z\to  M$ such that $\SSS_e(\zeta)<0$. Let $k\in\N$ be large enough so that $d(\zeta(\tfrac{i}{k}s),\zeta(\tfrac{i+1}{k}s))<\rho_\inj$ for all $i=0,...,k-1$. We denote by $\eta:\R/r\Z\to M$ the piecewise smooth curve such that, for some $0=r_{0}<...<r_{k-1}=r$, each restriction $\eta|_{[r_i,r_{i+1}]}$ is the unique local free-time action minimizers with energy $e$ joining $\eta(r_i)=\zeta(\tfrac{i}{k}s)$ and $\eta(r_{i+1})=\zeta(\tfrac{i+1}{k}s)$. Clearly $\SSS_e(\eta)\leq\SSS_e(\zeta)<0$ and, up to choosing $k$ large enough, the periodic curve $\eta$ is freely homotopic to $\zeta$. In particular, $\eta$ is null-homologous. Moreover, up to a generic perturbation of the vertices $\eta(r_{i})$, we can assume that $\eta$ has only finitely many self-intersections, all of which are double points.

\begin{figure}
\begin{center}
\begin{small}
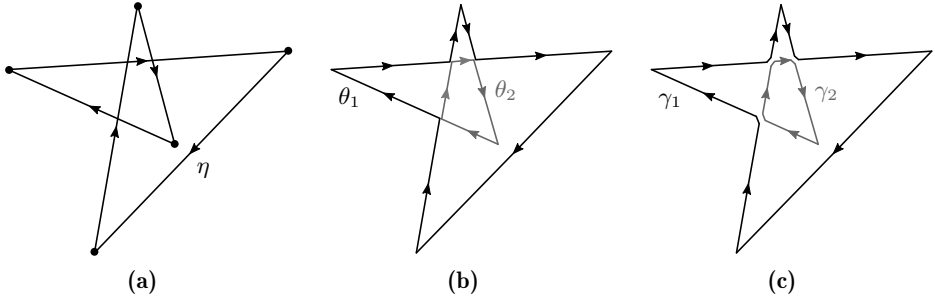 
\caption{\textbf{(a)}~A curve $\eta$ with three double points.
\textbf{(b)}~The rearrargement of $\eta$ as a multicurve $\ttheta=(\ttheta_1,\ttheta_2)$ whose only self-intersections are tangencies. \textbf{(c)}~The embedded multicurve $\ggamma=(\gamma_1,\gamma_2)$ obtained by chamfering the corners of $\ttheta$.}
\vspace{-3mm}
\label{f:cut_and_paste}
\end{small}
\end{center}
\end{figure}

We now produce a multicurve $\ttheta=(\theta_1,...,\theta_m)\in\ELMult_e(m)$, for some $m\leq n$, with the following properties: the total support of $\ttheta$ coincide with the support of $\eta$, its action is $\SSS_e(\ttheta)=\SSS_e(\eta)$, and the only self intersections of $\ttheta$ are tangencies. Such a multicurve $\ttheta$ is obtained starting from $\eta$, by transforming any crossing at a double point into a tangency in the unique way compatible with the orientation, as in the example in Figure~\ref{f:cut_and_paste}(a-b).

Since $\ttheta$ is obtained by rearranging $\eta$, its homology class is $[\ttheta]=[\eta]=0$ in $\Hom_1(M;\Z)$, that is, $\ttheta$ is a homological boundary. By chamfering the corners of $\ttheta$ (Figure~\ref{f:cut_and_paste}(c)), we can find an embedded multicurve $\ggamma\in\ELMult_e(m)$ that is arbitrarily close to $\ttheta$. In particular, we choose $\ggamma$ to be sufficiently close to $\ttheta$ so that $[\ggamma]=[\ttheta]=0$ in $\Hom_1(M;\Z)$ and $\SSS_e(\ggamma)<0$. By Lemma~\ref{l:boundaries}, the multicurve $\ggamma$ can be written as a finite union $\ggamma=\ggamma_1\cup...\cup\ggamma_h$, where each $\ggamma_j$ is a topological boundary. In particular, each $\ggamma_j$ belongs to $\Mult_e$. Since 
\begin{align*}
\SSS_e(\ggamma_1) + \ldots + \SSS_e(\ggamma_h) = \SSS_e(\ggamma) < 0,
\end{align*}
at least one such $\ggamma_j$ must satisfy $\SSS_e(\ggamma_j)<0$.
\end{proof}

\begin{lem}\label{l:minimal_topological_boundaries}
For each $e>e_0(L)$, the components of minimal boundaries with energy $e$ are (simple) Tonelli waists with energy $e$.
\end{lem}

\begin{proof}
Let $\gamma_i$ be a component of a minimal boundary $\ggamma$ with energy $e$, and $t_0\in\R$ be such that $\gamma_i$ is smooth at $t_0$. For all sufficiently large $n\in\N$, we have that $d(\gamma_i(t_0),\gamma_i(t_0+n^{-1}))<\rho_\inj$. We denote by $\zeta_n:[0,\tau_n]\to M$ the unique local free-time action minimizer with energy $e$ such that $\zeta_n(0)=\gamma_i(t_0)$ and $\zeta_n(\tau_n)=\gamma_i(t_0+n^{-1})$. As $n\to\infty$, we have that $\tau_n\to 0$ and $\dot\zeta_n(0)\to\dot\gamma_i(t_0)$. This readily implies that, for all $n$ large enough, $\zeta_n$ intersects the embedded multicurve $\ggamma$ only in $\gamma_i([t_0,t_0+n^{-1}])$. Fix now such a large enough $n\in\N$. We denote by $\gamma_i'$ the curve $\gamma_i$ with the portion $\gamma_i|_{[t_0,t_0+n^{-1}]}$ replaced by $\zeta_n$, and by $\ggamma'$ the multicurve $\ggamma$ with the component $\gamma_i$ replaced by $\gamma_i'$. Clearly, $\SSS_e(\ggamma')\leq\SSS_e(\ggamma)$, and this inequality is not strict if and only if the curves $\gamma_i|_{[t_0,t_0+n^{-1}]}$ and $\zeta_n$ coincide. Since $\ggamma$ is a minimal boundary with energy $e$ and $\ggamma'\in\tb$, we infer that $\ggamma\in\Mult_e\cap\tb$, that is, each component of $\ggamma$ is a piecewise smooth solution of the Euler-Lagrange equation with energy $e$. Now, arguing as in the proof of \cite[Lemma~2.7]{Asselle:2016qv}, we infer that each component $\gamma_i$ of $\ggamma$ is a simple periodic orbit with energy $e$. Assume by contradiction that one such $\gamma_i$ is not a Tonelli waist. Therefore, there exists a piecewise smooth $\gamma_i''$ that is arbitrarily $C^1$-close to $\gamma_i$ and satisfies $\SSS_e(\gamma_i')<\SSS_e(\gamma_i)$. We fix one such $\gamma_i''$ that is $C^1$-close enough to $\gamma_i$ so that it is a simple curve that does not intersect any other component of $\ggamma$. By replacing the component $\gamma_i$ of $\ggamma$ with $\gamma_i''$, we  thus obtain a multicurve $\ggamma''$ that is still a topological boundary and satisfies $\SSS_e(\ggamma'')<\SSS_e(\ggamma)$. This contradicts the fact that $\ggamma$ is a minimal boundary.
\end{proof}

\begin{proof}[Proof of Theorem~\ref{t:below_Mane} for $e<c_0(L)$]
Lemmas~\ref{l:negative_implies_waist} and~\ref{l:minima_are_negative} imply that there exists a multicurve $\ggamma\in\Mult_e$ such that $\SSS_e(\ggamma)=\inf\SSS_e|_{\Mult_e}<0$. Lemma~\ref{l:minima_are_embedded} further implies that $\ggamma$ is embedded (thus a topological boundary), and its components are Tonelli waists with energy $e$. Finally, Lemma~\ref{l:minimal_topological_boundaries} implies that minimal boundaries with energy $e$ belongs to $\Mult_e$, and in particular
\[  \inf_{\tb}\SSS_e = \inf_{\Mult_e}\SSS_e = \SSS_e(\ggamma). \]
Therefore, $\ggamma$ is a minimal boundary with energy $e$.
\end{proof}

%%%%%%%%%%%%%%%%%%%%%%

\subsection{Critical minimal boundaries}\label{s:Mane}
The very definition of the Ma\~n\'e critical value $c_0(L)$ implies that $\SSS_{c_0(L)}$ is non-negative over the space of null-homologous absolutely continuous periodic curves in $M$. This actually implies the following property, which holds for general closed manifolds $M$.

\begin{lem}\label{l:actioncrit}
Let $M$ be a closed manifold, and $L:\Tan M\to\R$ a Tonelli Lagrangian. The associated action functional $\SSS_{c_0(L)}$ is non-negative over the space of homological boundaries of $M$.
\end{lem}
\begin{proof}
Let us assume by contradiction that there exists a homological boundary $\ggamma=(\gamma_1,...,\gamma_m)$ with $\SSS_{c_0(L)}(\ggamma)<0$. For each $i=1,...,m-1$, we choose an absolutely continuous path $\zeta_i:[0,1]\to M$ such that $\zeta_i(0)=\gamma_i(0)$ and $\zeta_i(1)=\gamma_{i+1}(0)$. For each $n\in\N$, we define the loop
\[
\xi_n:=\gamma_1^n *\zeta_1* \gamma_2^n*\ldots*\zeta_{m-1}* \gamma_m^n*\overline\zeta_{m-1}*\overline\zeta_{m-2}*\ldots* \overline\zeta_1,
\]
where $\overline\zeta_i:[0,1]\to M$ denotes the reversed path $\overline\zeta_i(t)=\zeta_i(1-t)$ joining $\gamma_{i+1}(0)$ and $\gamma_{i}(0)$, $*$ denotes concatenation of paths, and the superscript $n$ denotes the $n$-th iteration of a loop. Notice that $\xi_n$ is null-homologous, for
\begin{align*}
[\xi_n]
&=
[\gamma_1^n] + \ldots+[\gamma_m^n]+\underbrace{[\zeta_{1}*\overline\zeta_{1}]}_{=0}+\ldots+\underbrace{[\zeta_{m-1}*\overline\zeta_{m-1}]}_{=0}
=n [\ggamma]=0.
\end{align*}
However, its action
\[
\SSS_{c_0(L)}(\xi_n)=n\, \SSS_{c_0(L)}(\ggamma)+\sum_{i=1}^{m-1}\SSS_{c_0(L)}(\zeta_i*\overline
\zeta_i),
\] 
is negative for $n$ large enough, contradicting the definition of $c_0(L)$.
\end{proof}

\begin{proof}[Proof of Theorem~\ref{t:below_Mane} for $e=c_0(L)$]
Fix an energy value $e_1\in(e_0(L),c_0(L))$, and consider the Tonelli injectivity radius $\rho_\inj=\rho_\inj(e_1)>0$ introduced in Section~\ref{s:below_Mane}. Its properties readily imply that every absolutely continuous periodic curve $\gamma:\R/p\Z\to M$, with $p>0$, that is entirely contained in a closed Riemannian ball of radius $\rho_\inj$ satisfies $\SSS_{e_1}(\gamma)>0$, and thus also $\SSS_{e}(\gamma)=\SSS_{e_1}(\gamma)+(e-e_1)p>0$ for all $e>e_1$. Here, $\rho_\inj$ is small enough so that all closed Riemannian balls of radius $\rho_\inj$ are contractible.

By \cite[Lemma~2.9]{Asselle:2016qv}, any piecewise smooth multicurve $\ggamma=(\gamma_1,...,\gamma_m)\in\overline{\tb(m)}$ with components of the form $\gamma_i:\R/p_i\Z\to M$ satisfies
\begin{align}\label{e:period_bound}
\sum_{i=1}^m p_i \leq \frac{\SSS_e(\ggamma)+\int_M |\diff\theta|}{e-e_0(L)},
\end{align}
where $\theta$ is the 1-form on $M$ given by $\theta_q(v)=\partial_vL(q,0)v$.

For each $e\in(e_1,c_0(L))$, let $\ggamma_e=(\gamma_{e,1},...,\gamma_{e,m_e})\in\Mult_e$ be the minimal boundary given by Theorem~\ref{t:below_Mane}, so that 
\begin{align}\label{e:inf}
 \SSS_e(\ggamma_e) = \inf_{\tb}\SSS_e = \inf_{\Mult_e}\SSS_e < 0.
\end{align}
We claim that no component of $\ggamma_e$ is contained in a Riemannian ball of radius $\rho_\inj$. Otherwise, after removing all such components $\gamma_{e,i}$ (which have  action $\SSS_e(\gamma_{e,i})>0$), we would be left with a multicurve $\ggamma'\subseteq\ggamma_e$ that is still a topological boundary, but satisfies $\SSS_e(\ggamma')<\SSS_e(\ggamma_e)$, contradicting~\eqref{e:inf}. In particular, every component $\gamma_{e,i}:\R/p_{e,i}\Z\to M$ has length at least $2\rho_\inj$ and period bounded from below as
\begin{align*}
p_{e,i} \geq p_{\min}:= 2\rho_\inj \min\big\{ |v|_q\ \big|\ E(q,v)\geq e_1 \big\}.
\end{align*}
By~\eqref{e:period_bound}, the total period of the multicurve $\ggamma_e$ can be bounded from above as
\begin{align*}
\sum_{i=1}^{m_e} p_{e,i} < p_{\max}:=\frac{\int_M|\diff\theta|}{e_1-e_0(L)}.
\end{align*}
Notice that $p_{\max}$ and $p_{\min}$ are independent of $e\in(e_1,c_0(L))$. 
The total length of the multicurve $\ggamma_e$ can be uniformly bounded as
\begin{align*}
 \sum_{i=1}^{m_e} \mathrm{length}(\gamma_{e,i}) <  p_{\max} \max\big\{ |v|_q\ \big|\ E(q,v)\leq c_0(L) \big\}.
\end{align*}
Therefore, the number $m_e$ of connected components of $\ggamma_e$ is uniformy bounded as
\begin{align*}
1\leq m_e\leq  \frac{p_{\max}}{2\rho_\inj} \max\big\{ |v|_q\ \big|\ E(q,v)\leq c_0(L) \big\}. 
\end{align*}
This, together with the pigeonhole principle, implies that there exists $m\in\N$ and  a monotone increasing sequence $e_\alpha\to c_0(L)$ such that each multicurve $\ggamma_{e_\alpha}$ has $m$ connected components. Since the lengths and the periods of the connected components of the $\ggamma_{e_\alpha}$'s are uniformly bounded from above and uniformly bounded away from zero, up to extracting a subsequence we have that $\ggamma_{e_\alpha}$ converges in the $C^\infty$-topology to some multicurve $\ggamma\in\Mult_{c_0(L)}$ as $\alpha\to\infty$. The action of $\ggamma$ satisfies
\begin{align*}
\SSS_{c_0(L)}(\ggamma) = \lim_{\alpha\to\infty} \SSS_{e_\alpha}(\ggamma_{e_\alpha})\leq0.
\end{align*}
Since $\ggamma$ is a homological boundary, Lemma \ref{l:actioncrit} implies that $\ggamma$ is a global minimizer of $\SSS_{c_0(L)}$ on the space of homological boundaries and has action $\SSS_{c_0(L)}(\ggamma)=0$. In particular, $\ggamma$ is a global minimizer of $\SSS_{c_0(L)}|_{\Mult_{c_0(L)}}$, and Lemma~\ref{l:minima_are_embedded} implies that $\ggamma$ is embedded, and thus a minimal boundary.
\end{proof}

Unlike on subcritical energies, minimal boundaries with energy $c_0(L)$ have the following decomposition property.

\begin{cor}\label{c:irreducible_minimal}
Let $M$ be an oriented closed surface of genus $g$, and $L:\Tan M\to\R$ a Tonelli Lagrangian with $e_0(L)<c_0(L)$. Every minimal boundary with energy $c_0(L)$ can be decomposed as the disjoint union of irreducible minimal boundaries with energy $c_0(L)$. In particular, there exist irreducible minimal boundaries with energy $c_0(L)$, which have at most $g+1$ components.
\end{cor}
\begin{proof}
Let $\ggamma$ be a minimal boundary at level $c_0(L)$, whose existence is guaranteed by Theorem~\ref{t:below_Mane}. By Lemma~\ref{l:boundaries}, the homological boundary $\ggamma$ can be decomposed as a disjoint union $\ggamma=\ggamma_1\cup...\cup\ggamma_k$, where each $\ggamma_i$ is an irreducible topological boundary. Since $\SSS_{c_0(L)}(\ggamma)=\SSS_{c_0(L)}(\ggamma_1)+...+\SSS_{c_0(L)}(\ggamma_m)=0$, Lemma~\ref{l:actioncrit} implies that $\SSS_{c_0(L)}(\ggamma_i)=0$ for all $i=1,...,m$.
Therefore, all the irreducible topological boundaries $\ggamma_i$ are minimal, and Lemma~\ref{l:irreducible} implies that each of them has at most $g+1$ components.
\end{proof}

The upper bound $g+1$ on the number of components of an irreducible minimal boundary with energy $c_0(L)$ is optimal. Indeed, as the following example shows, any given irreducible topological boundary is the unique minimal boundary with energy $c_0(L)$ for a suitable Tonelli Lagrangian $L$.

\begin{exm}
Let $M$ be a closed oriented surface, and $\ggamma=(\gamma_1,...,\gamma_{m})$ an irreducible topological boundary in $M$. We choose a Riemannian metric $g$ on $M$ with respect to which the components $\gamma_i:\R/p_i\Z\to M$ have unit speed. We denote by $|\cdot|$ the norm on vectors and covectors induced by $g$. We choose a 1-form $\theta$ on $M$ such that $|\theta_q|<1$ for all $q\in M\setminus\ggamma$, and $\theta_{\gamma_i(t)}=-g(\dot\gamma_i(t),\cdot)$ for all $i=1,...,m$ and $t\in\R/p_i\Z$. We define the Tonelli Lagrangian 
\[L:\Tan M\to\R, \qquad L(q,v)=\tfrac12g_q(v,v)+\theta_q(v),\] and claim that $\ggamma$ is the unique minimal boundary for $L$ with energy $c_0(L)$. Indeed, consider the dual Tonelli Hamiltonian $H:\Tan^*M\to\R$, $H(q,p)=\tfrac12|p-\theta_q|^2$.
By the minmax characterization \cite[Theorem~A]{Contreras:1998lr} of the Ma\~n\'e critical value, we have
\begin{align*}
c_0(L)\leq c(L)=\inf_{u\in C^\infty(M)} \max_{q\in M} H(q,\diff u(q))\leq \max_{q\in M}H(q,0)=\tfrac12.
\end{align*}
Since 
\begin{align*}
\SSS_{1/2}(\ggamma)=\int_0^p \Big(\tfrac12|\dot\zeta(t)| +\theta(\dot\zeta(t))\Big)\,\diff t + \tfrac12\, p 
= -\tfrac12p+\tfrac12p
=0,
\end{align*}
we have that $c_0(L)\geq\tfrac12$. Therefore, 
\begin{align*}
 c(L)=c_0(L)=\tfrac12>0=e_0(L).
\end{align*}
and $\ggamma$ is a minimal boundary with energy $c_0(L)$.
Assume now that $\zeta:\R/p\Z\to M$ is a periodic orbit with energy $c_0(L)$. Since the energy function associated to $L$ is 
$E:\Tan M\to\R$, $E(q,v)=\tfrac12g_q(v,v)$,
the curve $\zeta$ has unit speed. If we assume that $\zeta$ is not one of the components of $\ggamma$, we have $\theta(\dot\zeta)>-1$, and therefore
\begin{align*}
\SSS_{1/2}(\zeta)=\int_0^p \Big(\tfrac12|\dot\zeta(t)| +\theta(\dot\zeta(t))\Big)\,\diff t + \tfrac12 p > -\tfrac12 p + \tfrac12 p = 0.
\end{align*}
This shows that no curve other than the $\gamma_i$'s is the component of a minimal boundary with energy $c_0(L)$.
\hfill\qed
\end{exm}

\subsection{The invariant sets $\Graph_e(L)$}\label{s:graph}
Let $M$ be a closed oriented surface, and $L:\Tan M\to\R$ a Tonelli Lagrangian such that $e_0(L)<c_0(L)$. Let us rephrase the statement of Theorem~\ref{t:graph} in a way that is more convenient for its proof:
\vspace{5pt}

\noindent
\textbf{Theorem~\ref{t:graph} (rephrased)}. \textsl{If $\ggamma=(\gamma_1,...,\gamma_m)$ and $\zzeta=(\zeta_1,...,\zeta_n)$ are two minimal boundaries with energy $e\in(e_0(L),c_0(L)]$ for $L$ such that $\gamma_i(r)=\zeta_j(s)$ for some $i\in\{1,...,m\}$, $j\in\{1,...,n\}$, and $r,s\in\R$, then $\gamma_i(r+t)=\zeta_j(s+t)$ for all $t\in\R$.}

\begin{proof}
We fix an energy $e>e_0(L)$ and an arbitrary Riemannian metric on $M$ which induces a distance $\dist:M\times M\to[0,\infty)$. We already introduced the Tonelli injectivity radius $\rho_\inj>0$ in Section~\ref{s:below_Mane}. Actually, by \cite[Lemma~2.3 and Cor.~2.5]{Asselle:2016qv}, $\rho_\inj$ can be chosen so that, for each closed Riemannian ball $B\subset M$ of radius $\rho_\inj$, the following properties hold:
\begin{itemize}
\item[(i)] for all distinct curves $\gamma:[a,b]\to B$ and $\zeta:[c,d]\to B$ that are portions of some orbits of the Lagrangian system of $L$ with energy $e$ satisfying $\gamma(t_1)=\zeta(s_1)$ and $\gamma(t_2)=\zeta(s_2)$ for some $t_1,t_2\in[a,b]$ and $s_1,s_2\in[c,d]$, we have $t_1\leq t_2$ if and only if $s_1\geq s_2$;

\item[(ii)] any absolutely continuous periodic curve $\gamma:\R/p\Z\to B$ with $p>0$ is contractible and has action $\SSS_e(\gamma)>0$.
\end{itemize}

Let us now consider two minimal boundaries $\ggamma$ and $\zzeta$ as in the statement. Since they are in particular smooth embedded multicurves (by Lemma~\ref{l:minimal_topological_boundaries}), we can find $\rho=\rho(\ggamma,\zzeta)\in(0,\rho_\inj)$ small enough such that the following property holds
\begin{itemize}
 
\item[(iii)] for each closed Riemannian ball $B\subset M$ of radius $\rho$,  the intersections  $\ggamma\cap B$ and $\zzeta\cap B$ are diffeomorphic to a compact interval (and possibly the interval reduces to a single point when a multicurve intersects $B$ with an external tangency).
 
\end{itemize}
Let $\Sigma_1,\Sigma_2\subset M$ be the open subsets whose oriented boundaries are $\partial\Sigma_1=\ggamma$ and $\partial\Sigma_2=\zzeta$. A priori, the open sets $\Sigma_1\cap\Sigma_2$ and $M\setminus(\overline{\Sigma_1\cup\Sigma_2})$ may have infinitely many connected components (this happens if $\ggamma$ and $\zzeta$ have non-isolated intersections). We denote by $\Pi_1\subset \Sigma_1\cap\Sigma_2$ and $\Pi_2\subset M\setminus(\overline{\Sigma_1\cup\Sigma_2})$ the union of their ``small'' connected components, where a connected component is considered small if it is contained in a Riemannian ball of radius $\rho$. Points (i) and (iii) imply that every connected component of $\Pi_1$ and $\Pi_2$ is an open ball whose oriented boundary is a curve obtained by concatenating an interval contained in $\ggamma$ with an interval contained in $\zzeta$.

Notice that the open sets $\Sigma_1\cap\Sigma_2\setminus\Pi_1$ and $M\setminus(\overline{\Sigma_1\cup\Sigma_2}\cup \Pi_2)$ have finitely many connected components. Indeed, any such connected component $\Upsilon$ is not contained in any  Riemannian ball of radius $\rho$, and therefore its boundary $\partial\Upsilon$ has length at least $2\rho$. This gives
\begin{align*}
\# \pi_0(\Sigma_1\cap\Sigma_2\setminus\Pi_1) + \#\pi_0(M\setminus(\overline{\Sigma_1\cup\Sigma_2}\cup \Pi_2)) 
\leq \frac{\mathrm{length}(\ggamma) + \mathrm{length}(\zzeta)}{2\rho}.
\end{align*}

The oriented boundaries $\ttheta:=\partial(\Sigma_1\cap\Sigma_2\setminus\Pi_1)$ and $\eeta:=\partial(-M\setminus(\overline{\Sigma_1\cup\Sigma_2}\cup \Pi_2))$, where the minus sign denotes the reverse orientation, are obtained by concatenating  finitely many subintervals of the multicurve $\ggamma\cup\zzeta$ (here, each subinterval of $\ggamma\cup\zzeta$ contributes to at most one subinterval of $\ttheta\cup\eeta$). Indeed, assume that there is a continuous curve of the form $\sigma:[0,t_4]\to\partial(\Sigma_1\cap\Sigma_2)$ or $\sigma:[0,t_4]\to\partial(M\setminus(\overline{\Sigma_1\cup\Sigma_2}))$ such that, for some $0<t_1<t_2<t_3<t_4$, $\sigma|_{[0,t_1]}$ and $\sigma|_{[t_2,t_3]}$ are subintervals of $\ggamma$,  $\sigma|_{[t_1,t_2]}$ and $\sigma|_{[t_3,t_4]}$ are subintervals of $\zzeta$, and that $\sigma|_{[t_1,t_2]}$ and $\sigma|_{[t_2,t_3]}$ have length at most $\rho$. The curve $\sigma|_{[t_1,t_3]}$ is contained in a closed Riemannian ball $B$ of radius $\rho$. Properties (i) and (iii) imply that $\sigma(t_1)=\sigma(t_3)$, that is, $\sigma$ winds around the closed curve $\sigma|_{[t_1,t_3]}$ of length at most $2\rho$, and therefore it is the boundary of a connected component of $\Pi_1$ or $\Pi_2$.

\begin{figure}
\begin{center}
\begin{footnotesize}
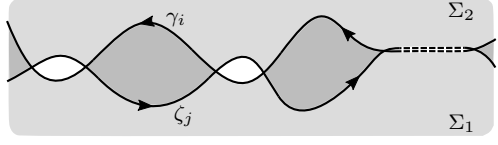 
\caption{Example of intersections between the multicurves $\ggamma$ and $\zzeta$. The small white regions are components of $\Pi_2$. The support of the curves $\gamma_i$ and $\zeta_j$ coincide with opposite orientation in the dashed subintervals, which are contained in $\partial(\overline{\Sigma_1\cup\Sigma_2})\setminus\partial(\Sigma_1\cup\Sigma_2)$.}
\label{f:extra_portions}
\end{footnotesize}
\end{center}
\end{figure}

The portions of $\ggamma$ and $\zzeta$ that are not employed for building $\ttheta$ and $\eeta$ can be of two kinds: those that form the oriented boundaries $\partial\Pi_1$ and $-\partial\Pi_2$, and those who form ``thin'' closed curves in $\partial(\overline{\Sigma_1\cup\Sigma_2})\setminus\partial(\Sigma_1\cup\Sigma_2)$, see Figure~\ref{f:extra_portions}. We denote by $\gamma_i':[0,a_i]\to M$, for $i=1,...,k$, and $\zeta_j':[0,b_j]\to M$, for $j=1,...,l$, the portions of $\ggamma$ and $\zzeta$ respectively that are not employed for building $\ttheta$ and $\eeta$. We denote by $I\subset M$ the set of intersections between $\ggamma$ and $\zzeta$. Notice that $\gamma_i'(0),\gamma_i'(a_i)\in I$, and moreover each $t\in[0,a_i]$ is contained in a closed interval $[t',t'']\in[0,a_i]$ such that $\gamma_i'(t'),\gamma_i'(t'')\in I$ and $\dist(\gamma_i'(t'),\gamma_i'(t''))\leq 2\rho$. An analogous statement holds for the curves $\zeta_j'$. Properties (i) and (iii) imply that $k=l$ and, up to reordering the curves $\zeta_1',...,\zeta_l'$, we can find finite sequences 
\begin{align*}
 0=a_{i,0}<a_{i,1}<...<a_{i,k_i}=a_i,\\
 0=b_{i,k_i}<b_{i,k_i-1}<...<b_{i,0}=b_i
\end{align*}
such that, for all $i=1,...,k$,
\begin{align*}
 \gamma_i'(a_{i,j}) = \zeta_i'(b_{i,j}),\qquad\qquad & \forall j=1,...,k_i,\\
 \dist(\gamma_i'(a_{i,j}),\gamma_i'(a_{i,j+1}))\leq 2\rho,\qquad & \forall j=1,...,k_i-1.
\end{align*}
Therefore, for each $i\in\{1,...,k\}$ and $j\in\{1,...,k_i-1\}$, the restrictions $\gamma_i'|_{[a_{i,j},a_{i,j+1}]}$ and $\zeta_i'|_{[b_{i,j+1},b_{i,j}]}$ are contained in a same closed Riemannian ball of radius $\rho$, and can be concatenated to form a loop. This, together with property~(ii), implies that 
\begin{align}
\label{e:extra_positive_action}
\sum_{i=1}^{k}\Big( \SSS_e(\gamma_i') + \SSS_e(\zeta_i') \Big) \geq 0
\end{align}
and such an inequality is strict unless $k=0$, that is, unless all the portions of $\ggamma$ and $\zzeta$ are employed for building $\ttheta$ and $\eeta$. This implies that
\begin{equation}
\label{e:two_minima}
\begin{split}
\SSS_e(\ttheta) + \SSS_e(\eeta)
 & \leq
\SSS_e(\ttheta) + \SSS_e(\eeta) + \sum_{i=1}^{k}\Big( \SSS_e(\gamma_i') + \SSS_e(\zeta_i') \Big)\\
 & = \SSS_e(\ggamma) + \SSS_e(\zzeta) \\
 & = 2 \inf_{\tb}\SSS_e \\
 & = 2 \inf_{\Mult_e}\SSS_e,
\end{split}
\end{equation}
Since $\ttheta$ and $\eeta$ are formed by concatenating finitely many subintervals of $\ggamma$ and $\zzeta$, in particular  $\ttheta,\eeta\in\Mult_e$.  The inequality~\eqref{e:two_minima} thus forces
\begin{align*}
\SSS_e(\ttheta) = \SSS_e(\eeta) = \SSS_e(\ggamma) = \SSS_e(\zzeta) = \inf_{\Mult_e}\SSS_e,
\end{align*}
and $k=0$. In particular $\Pi_1=\Pi_2=\varnothing$, and
\begin{align*}
 \ttheta =\partial(\Sigma_1\cap\Sigma_2),
 \qquad
 \eeta=\partial(\Sigma_1\cup\Sigma_2).
\end{align*}
By Lemma~\ref{l:minima_are_embedded}, $\ttheta$ and $\eeta$ are embedded multicurves whose components are simple Tonelli waists with energy $e$. Since $\ttheta$ and $\eeta$  were constructed by concatenating subintervals of $\ggamma$ and $\zzeta$, each  component of $\ttheta$ and $\eeta$ is a component of $\ggamma$ or $\zzeta$.

\begin{figure}
\begin{center}
\begin{footnotesize}
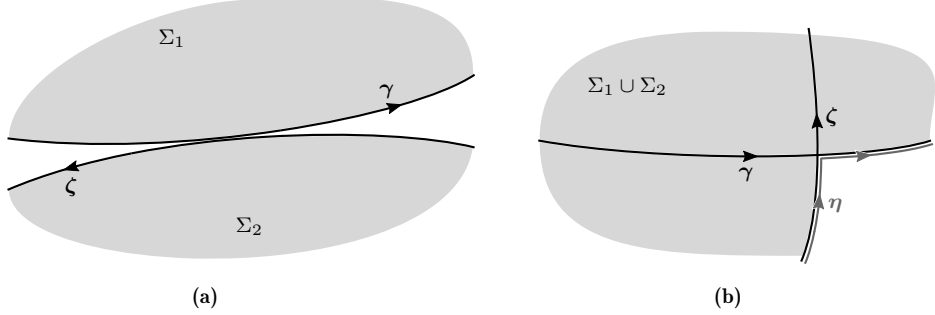 
\caption{\textbf{(a)} Tangency between $\ggamma$ and $\zzeta$.
\textbf{(b)} Transversal intersection between $\ggamma$ and $\zzeta$.}
\label{f:intersection}
\end{footnotesize}
\end{center}
\end{figure}

Assume now that $\gamma_i(r)=\zeta_j(s)$ for some $i\in\{1,...,m\}$, $j\in\{1,...,n\}$, and $r,s\in\R$, but $\gamma_i(r+t)\neq\zeta_j(s+t)$ for some $t\in\R$. Therefore, $\gamma_i$ and $\zeta_j$ intersect, but are not the same curve. This gives a contradiction: if the only intersections between $\gamma_i$ and $\zeta_j$ were tangencies, the multicurve $\eeta=\partial(\Sigma_1\cup\Sigma_2)$ would not be embedded (see Figure~\ref{f:intersection}(a)); otherwise, if $\gamma_i$ and $\zeta_j$ intersected in a topologically essential way (for instance transversally), the components of $\eeta$ would not be orbits of the Lagrangian system of $L$ (in case of a transvere intersection, they would not even be smooth, see Figure~\ref{f:intersection}(b)).
\end{proof}

Let us now focus on the energy level $c_0(L)$, and consider the setting of invariant measures introduced before Theorem~\ref{t:graph}. We fix a minimal measure $\mu\in\Meas_{\min}(L)$. An argument due to Haefliger \cite[Prop.~2.1]{Contreras:2004lv} implies that the support of $\mu$ is foliated by periodic orbits of the Euler-Lagrange flow $\phi_L^t$ in the energy level $E^{-1}(c_0(L))$. For each $(q,v)\in\supp(\mu)$, we denote by $p(q,v)$ the period of the orbit $t\mapsto\Gamma_{(q,v)}(t):=\phi_L^t(q,v)$, and we denote by $\gamma_{(q,v)}:=\pi\circ\Gamma_{(q,v)}$ its base projection, which is a simple periodic curve in $M$ according to Mather's graph theorem.

\begin{lem}\label{l:neaby_periodic_orbits}
For each $(q,v)\in\mathrm{supp}(\mu)$ and $\epsilon>0$, there exists a neighborhood $U\subset \Tan M$ of $(q,v)$ such that, for every $(q',v')\in\mathrm{supp}(\mu)\cap U$, the periodic orbit $\gamma_{(q',v')}$ has minimal period $p(q',v')\in(p(q,v)-\epsilon,p(q,v)+\epsilon)$ and satisfies $[\gamma_{(q',v')}]=[\gamma_{(q,v)}]$ in $\Hom_1(M;\Z)$.
\end{lem}

\begin{proof}
We fix $(q,v)\in\mathrm{supp}(\mu)$, and an embedded open interval $I\subset M$ intersecting the curve $\gamma_{(q,v)}$ transversely at $q$. Mather's graph theorem \cite{Mather:1991xd} also says that the inverse of the injective projection $\pi:\supp(\mu)\to\pi(\supp(\mu))\subset M$ is a Lipschitz map. Therefore, up to shrinking $I$, we can suppose that for each $q'\in I\cap\pi(\supp(\mu))$, if we denote by $v'\in\Tan_{q'}M$ the unique tangent vector such that $(q',v')\in\supp(\mu)$, $v'$ is transverse to $I$ and points to the same side of $I$ as the vector $v$. Up to choosing a smaller interval $J\subseteq I$ around $q$, there is a well defined continuous hitting-time function
$\tau  :J\cap\pi(\supp(\mu))\to(0,\infty)$ given by
$$\tau(q')  :=\inf\Big\{t>0\ \Big|\ (q',v')\in\supp(\mu),\ \gamma_{(q',v')}(t)\in I\Big\}.$$
Up to further shrinking $J$ around $q$ the following holds: for every $(q',v')\in \supp(\mu)\cap\pi^{-1}(J)$, the curve $\gamma_{(q',v')}|_{[0,\tau(q')]}$ is either closed or almost closed.
In the latter case, Mather's graph theorem implies that the points $q'$ and $\gamma_{(q',v')}(\tau(q'))$ lie in the same connected component of $I\setminus \{q\}$. 
Therefore, if we close up $\gamma_{(q',v')}|_{[0,\tau(q')]}$ by concatenating it with a little segment $\zeta_{(q',v')}\subset I$, we obtain a closed curve $\gamma_{(q',v')}|_{[0,\tau(q')]}*\zeta_{(q',v')}$ that, together with $\gamma_{(q,v)}$, form the (non-oriented) boundary of an annulus $A\subset M$. We claim that $\tau(q')=p(q',v')$, that is, the segment $\zeta_{(q',v')}$ is always reduced to a point, and thus the curve $\gamma_{(q',v')}$ closes up at the hitting time $\tau(q')$. Indeed, assume by contradiction that $\zeta_{(q',v')}$ is not a point curve. Then, $\gamma_{(q',v')}(t)$ belongs to the interior of the annulus $A$ for either $t>\tau(q')$ sufficiently close to $\tau(q')$ or $t<0$ sufficiently close to $0$. Consider the first case (Figure~\ref{f:annulus}), the second one being completely analogous. Since the curve $\gamma_{(q',v')}$ is periodic and does not have self-intersections nor intersections with $\gamma_{(q,v)}$, the exit time 
\[t':=\sup\Big\{t>\tau(q')\ \Big|\ \gamma_{(q',v')}(s)\in A, \ \forall s\in(\tau(q'),t)\Big\}\] 
is finite and $\gamma_{(q',v')}(t')$ lies in the interior of $\zeta_{(q',v')}$. However, this violates the above mentioned transversality property implied by Mather's graph theorem, which forces $\dot\gamma_{(q',v')}(t')$ to point in the interior of $A$.

\begin{figure}
\begin{center}
\begin{small}
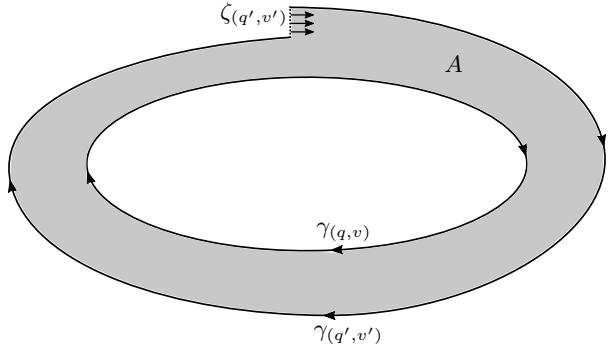 
\caption{The annulus $A$.}
\vspace{-3mm}
\label{f:annulus}
\end{small}
\end{center}
\end{figure}

Now, for each $\epsilon>0$, we can choose a neighborhood $U\subset\Tan M$ of $(q,v)$ that is small enough so that:
\begin{itemize}
\item $\tau(q')\in (p(q,v)-\epsilon,p(q,v)+\epsilon)$ for each $(q',v')\in U\cap\supp(\mu)\cap\pi^{-1}(J)$,
\item the periodic curve $\gamma_{(q',v')}(t)=\pi\circ\phi_L^t(q,v)$ hits $J$ in both positive and negative time $t$,   for each $(q',v')\in \mathrm{supp}(\mu)\cap U$.
\end{itemize}
The neighborhood $U$ has the desired properties.
\end{proof}

A Theorem of Ma\~n\'e implies that the measure $\mu\in\Meas_{\min}(L)$ also minimizes the action functional $\SSS$ among (not necessarily invariant) holonomic probability measures with zero rotation vector. This implies that each periodic orbit $\gamma_{(q,v)}$, for $(q,v)\in\supp(\mu)$, is a Tonelli waist with energy $c_0(L)$. Moreover, we have the following lemmas. 

\begin{lem}
\label{l:Mather1}
Every periodic orbit $\gamma_1$ of $L$ whose lift $(\gamma_1,\dot \gamma_1)$ is contained in $\supp(\mu)$ is a component of topological boundary $\ggamma=(\gamma_1,...,\gamma_m)$ whose lift to $\Tan M$ is contained in $\supp(\mu)$.
\end{lem}
\begin{proof}
If $[\gamma_1]=0$, then $\gamma_1$ is itself a topological boundary, and the lemma follows. Therefore, in the remainder of the proof, we consider the case where $[\gamma_1]\neq0$. Since $\supp(\mu)$ is compact, by Lemma~\ref{l:neaby_periodic_orbits} we can decompose it as a finite disjoint union $\supp(\mu)=K_1\cup...\cup K_n$ such that, for each $i=1,...,n$, the subset $K_i$ is open and closed in $\supp(\mu)$, and there exists $h_i\in\Hom_1(M;\Z)$ such that $[\gamma_{(q,v)}]=h_i$ for all $(q,v)\in K_i$. If needed, we relabel the sets $K_i$ so that the lift $(\gamma_1,\dot \gamma_1)$ is contained in $K_1$.
We denote by $\chi_i:\Tan M\to[0,1]$ the characteristic function of $K_i$, that is,
\begin{align*}
\chi_i(q,v)=
\left\{
  \begin{array}{@{}lll}
    1, &  & \mbox{if }(q,v)\in K_i, \vspace{2pt} \\ 
    0, &  & (q,v)\in \Tan M\setminus K_i. 
  \end{array}
\right.
\end{align*}
Since the sets $K_i$ are open in $\supp(\mu)$, we have $\mu(K_i)>0$, and therefore we can introduce the probability measures $\mu_i=\mu(K_i)^{-1}\chi_i\mu$ on $K_i$. Since 
\begin{align*}
\mu=\mu(K_1) \mu_1 + \ldots + \mu(K_n)\mu_n,
\end{align*}
we have
\begin{align}
\label{e:rho_mu}
0=\rho(\mu)= \mu(K_1)\rho(\mu_1) +\ldots+ \mu(K_n)\rho(\mu_n).
\end{align}
For each $1$-form $\eta$ on $M$, Birkhoff's ergodic theorem implies that
\begin{align*}
\langle\eta,\mu(K_i)\rho(\mu_i)\rangle
&=
\int_{K_i}\!\!\! \eta_x(v)\,\diff\mu(q,v)\\
&=
\int_{K_i} p(q,v)^{-1}\left(\int_{\gamma_{(q,v)}} \!\!\! \eta\right)\diff\mu(q,v)\\
&=
\langle\eta,h_i\rangle
\underbrace{\int_{K_i} p(q,v)^{-1}\diff\mu(q,v)}_{=:a_i}.
\end{align*}
and therefore $\mu(K_i)\rho(\mu_i)=a_i h_i\in\Hom_1(M;\R)$, where $a_i>0$. Thus, Equation~\eqref{e:rho_mu} implies that
$a_1 h_1 +...+ a_n h_n=0$.
Since the classes $h_i$ are contained in $\Hom_1(M;\Z)$, the previous equality can also be satisfied with positive integer coefficients, i.e.
\begin{align*}
b_1h_1 +\ldots + b_nh_n=0
\end{align*}
for some integer coefficients $b_i>0$. Let us choose an arbitrary embedded homological boundary $\zzeta$ having precisely $b_i$ components with homology $h_i$, for all $i=1,...,n$. By Lemma~\ref{l:boundaries}, $\zzeta$ contains an irreducibile topological boundary $\zzeta'=(\zeta_1',...,\zeta_m')$ such that $[\zeta_1']=h_1$. By Lemma~\ref{l:irreducible}, the classes $[\zeta'_j]$ are pairwise distinct. Therefore, up to relabeling the homology classes $h_2,...,h_n$, we have that $[\zeta_j']=h_{j}$ for all $j=1,...,m$. We conclude that
\begin{align}\label{e:hi}
0=[\zzeta']=h_1+\ldots+h_m.
\end{align}
Now, for each $i=2,...,m$, we choose an arbitrary $(q_i,v_i)\in K_i$, and we consider the corresponding periodic orbit $\gamma_i:=\gamma_{(q_i,v_i)}$. Since $[\gamma_i]=h_i$, Equation~\ref{e:hi} implies that the multicurve $\ggamma:=(\gamma_1,...,\gamma_m)$ is a homological boundary. Since $\ggamma$ is irreducible, it is a topological boundary.
\end{proof}

\begin{lem}
\label{l:Mather2}
Every topological boundary whose lift to $\Tan M$ is contained in the Mather set $\Mather_0(L)$ is a minimal boundary with energy $c_0(L)$.
\end{lem}
\begin{proof}
The so-called Mather's alpha function $\alpha([\eta])=c(L-\eta)$, where $\eta$ is a closed 1-form on $M$ and $[\eta]\in H^1(M;\R)$, is convex and superlinear, see \cite[Theorem~2-6.4]{Contreras:1999fm}. We denote by $\eta$ a closed 1-form whose cohomology class $[\eta]$ is a global minimizer of $\alpha$. A result of Paternain-Paternain \cite[Theorem~1.1]{Paternain:1997ud} implies that $c_0(L)=c(L-\eta)$.
Let $\Meas(L)$ be the set of Borel probability measures $\mu$ on the tangent bundle $\Tan M$ that are invariant by the Euler-Lagrange flow $\phi_L^t$ (here, unlike in the definition of $\Meas_0(L)$, we do not require $\mu$ to have zero rotation vector). Notice that $\Meas(L)=\Meas(L-\eta)$, since the Euler-Lagrange flows $\phi_L^t$ and $\phi_{L-\eta}^t$ coincide. The Ma\~n\'e critical value $c(L)$ can be characterized as
\begin{align*}
c(L)  =-\inf_{\mu\in\Meas(L)}\int_{\Tan M} L\,\diff\mu.
\end{align*}
For each $(q,v)\in\Mather_0(L)$, we denote by $\mu_{(q,v)}\in\Meas(L)$ the invariant probability measure supported on the periodic orbit $\gamma_{(q,v)}$, i.e.
\begin{align*}
\int_{\Tan M} f\,\diff\mu_{(q,v)} = \frac{1}{p(q,v)} \int_0^{p(q,v)} f\circ\phi_L^t(q,v)\,\diff t,
\qquad
\forall f\in C^0(\Tan M).
\end{align*}
A result of Fathi-Giuliani-Sorrentino \cite[Lemma~3.5]{Fathi:2009fu} implies that, for every $(q,v)\in\Mather_0(L)$, the probability measure $\mu_{(q,v)}$ satisfies
\begin{align*}
\int_{\Tan M} (L-\eta)\,\diff\mu_{(q,v)} 
= 
\inf_{\mu\in\Meas(L)} \int_{\Tan M} (L-\eta)\,\diff\mu
=
-c(L-\eta)
=
-c_0(L).
\end{align*}
If we rephrase this identity in terms of the periodic orbit $\gamma_{(q,v)}$, we obtain that
\begin{align}
\label{e:action_in_Mather_set}
\SSS_{c_0(L)}(\gamma_{(q,v)}) = \int_{\gamma_{(q,v)}} \!\!\! \eta = \langle\eta,[\gamma_{(q,v)}]\rangle,
\qquad
\forall (q,v)\in\Mather_0(L).
\end{align}
Now, let $\ggamma$ be a topological boundary whose lift of each component $(\gamma_i,\dot\gamma_i)$ is contained in the Mather set $\Mather_0(L)$. Since $[\ggamma]=0$ in $\Hom_1(M;\Z)$, Equation~\eqref{e:action_in_Mather_set} implies that
\begin{align*}
\SSS_{c_0(L)}(\ggamma)=\langle\eta,[\ggamma]\rangle = 0.
\end{align*}
Therefore, Lemma~\ref{l:actioncrit} implies that $\ggamma$ is a minimal boundary with energy $c_0(L)$.
\end{proof}
We can now provide the characterization of the Mather set $\Mather_0(L)$ in terms of minimal boundaries.

\begin{prop}\label{p:Mather}
Let $M$ be an oriented closed surface, and $L:\Tan M\to\R$ a Tonelli Lagrangian with $e_0(L)<c_0(L)$. Then $\Graph_{c_0(L)}(L)=\Mather_0(L)$. 
\end{prop}

\begin{proof}
The inclusion $\Mather_0(L)\subseteq\Graph_{c_0(L)}(L)$ follows immediately from Lemmas~\ref{l:Mather1} and~\ref{l:Mather2}. Conversely, let $\ggamma=(\gamma_1,...,\gamma_m)$ be a minimal boundary with energy $c_0(L)$, and $p_1,...,p_m$ the periods of its components. From Theorem \ref{t:below_Mane} we know that $\SSS_{c_0(L)}(\ggamma)=0$. The minimal boundary $\ggamma$ defines an invariant measure $\mu_{\ggamma}\in\Meas(L)$ by
\begin{align*}
\int_{\Tan M} f\, \diff\mu_{\ggamma} = \frac{1}{p}\sum_{i=1}^m \int_0^{p_i}\! f(\gamma_i(t),\dot\gamma_i(t))\,\diff t,\quad \forall f\in C^0_{\mathrm{c}}(\Tan M),
\end{align*}
where $p:=p_1+...+p_m$. Clearly, $\mu_{\ggamma}$ has zero rotation vector and action given by 
\[
\SSS(\mu_{\ggamma}) = \frac{1}{p} \big( \SSS_{c_0(L)}(\ggamma) - p\,c_0(L) \big) = -c_0(L) = \inf_{\mu\in\Meas_0(L)}\int_{\Tan M}L\,\diff\mu,
\]
where in the last equality we used the chararacterization of $c_0(L)$ from \cite{Paternain:1997ud}. This implies that the tangent lift of $\ggamma$ is contained in the Mather set $\Mather_0(L)$, and therefore $\Graph_{c_0(L)}(L)\subseteq\Mather_0(L)$.
\end{proof}

%%%%%%%%%%%%%%%%%%%%%%%%%%%%%%%%

\subsection{Locally-minimal boundaries on supercritical energies}
\label{s:just_above}

Let us quickly recall the technicalities of the functional setting of the free-period action $\SSS_e$. For our purpose, it is enough to define $\SSS_e$ over the space of $W^{1,2}$ periodic curves with arbitrary period. Indeed, it is well known that the local minimizers in this setting are precisely the local minimizers in the absolutely continuous setting. This space of curves is formally given by $\W(\R/\Z,M)\times(0,\infty)$, where a pair $(\Gamma,p)$ in this product is identified with the $p$-periodic curve $\gamma(t)=\Gamma(t/p)$, and as usual we will simply write $\gamma=(\Gamma,p)$. The functional $\SSS_e$ is lower semicontinuous on $\W(\R/\Z,M)\times(0,\infty)$, and certain $\gamma=(\Gamma,p)\in\W(\R/\Z,M)\times(0,\infty)$ may have action $\SSS_e(\gamma)=\infty$. When working on a bounded energy range $[e_1,e_2]\subset\R$, as is the case in Theorem~\ref{t:just_above}, a way to gain more regularity is to modify the Lagrangian $L$ on $E^{-1}[e_2+1,\infty)$ in order to make it fiberwise quadratic outside a compact set of the tangent bundle $\Tan M$.
A construction of such a modified Lagrangian $L'$ can be found in, e.g., \cite[Prop.~18]{Contreras:2000zz}. By \cite[Lemma~19]{Contreras:2000zz}, we have $c(L)=c(L')$. For each $e\in[e_1,e_2]$ the free-period action functional $\SSS_e'$ of $L'$ is real-valued, $C^{1,1}$, and satisfies the Palais-Smale condition on subsets of the form $\W(\R/\Z,M)\times[p_-,p_+]$. Moreover it has the same critical points and local minimizers as the original $\SSS_e$ (see Lemma~\ref{l:remark_local_minimizers} in the Appendix). For this reason, in the rest of this subsection we can assume without loss of generality that the Tonelli Lagrangian $L$ is fiberwise quadratic outside a compact set of $\Tan M$.

The following general statement, which is actually valid on any closed configuration space $M$, will be needed to show that the periodic orbits provided by Theorem~\ref{t:just_above} are simple.

\begin{lem}\label{l:simple_orbits}
Let $M$ be a closed manifold, and $L:\Tan M\to\R$ a Tonelli Lagrangian. For every simple periodic orbit $\gamma=(\Gamma,p)\in W^{1,2}(\R/\Z,M)\times(0,\infty)$ of the Lagrangian system of $L$, there exists a neighborhood $\mathcal U\subset W^{1,2}(\R/\Z,M)\times(0,\infty)$ such that any periodic orbit of the Lagrangian system of $L$ in $\mathcal U$ is simple as well.
\end{lem}

\begin{proof}
By \cite[Theorem~2.7.4]{Fathi:2008xl}, for each energy value $e>e_0(L)$ there exists $\tau=\tau(e)>0$ small enough such that, for all $q\in M$ and $t\in(0,\tau)$, the map $\pi\circ\phi_L^t|_{\Tan_q M\cap E^{-1}(-\infty,e)}$ is a diffeomorphism onto an open neighborhood of $q$. Moreover, by \cite[Theorem~3.6.1]{Fathi:2008xl}, for all $(q,v)\in M$ with $E(q,v)<e$ and $\sigma\in(0,\tau)$, the curve 
\begin{align*}
 \gamma:[0,\sigma]\to M,\qquad\gamma(t)=\pi\circ\phi_L^t(q,v)
\end{align*}
is a strict action minimizer: if $\zeta:[0,\sigma]\to M$ is any other absolutely continuous curve such that $\zeta(0)=\gamma(0)$ and $\zeta(\sigma)=\gamma(\sigma)$, then 
\begin{align}\label{e:strict}
 \int_0^\sigma L(\gamma(t),\dot\gamma(t))\,\diff t
 <
 \int_0^\sigma L(\zeta(t),\dot\zeta(t))\,\diff t.
\end{align}

Let $\gamma_n=(\Gamma_n,p_n)$ be a sequence of periodic orbits that converges to the simple one $\gamma=(\Gamma,p)$ in $W^{1,2}(\R/\Z,M)\times(0,\infty)$ as $n\to\infty$. Let $e_n:=E(\gamma_n(0),\dot\gamma_n(0))$. We claim that there exists $e\in\R$ such that $e_n<e$ for all $n\in\N$. Otherwise, up to passing to a subsequence, we would have $e_n\to\infty$. This would imply that $|\dot\gamma_n(t)|\to\infty$ for all $t\in\R$, where $|\cdot|$ denotes an arbitrary Riemannian norm. Therefore, since the periods $p_n$ are bounded away from zero, we would obtain that 
$
\length(\gamma_n) = \int_0^{p_n} |\dot\gamma_n(t)|\,\diff t\to\infty,
$
which is impossible, for $\length(\gamma_n)\to\length(\gamma)$.

Let $\tau=\tau(e)>0$ be the constant introduced above, so that for all $q\in M$ and $t\in(0,\tau)$ the map $\pi\circ\phi_L^t|_{\Tan_q M\cap E^{-1}(-\infty,e)}$ is a diffeomorphism onto an open neighborhood of $q$. Since, for all $t\in(0,\tau)$, we have that $\gamma_n(t)=\pi\circ\phi_L^t(\gamma_n(0),\dot\gamma_n(0))\to\gamma(t)$ as $n\to\infty$, we infer that $\dot\gamma_n(0)\to\dot\gamma(0)$. Therefore, $\Gamma_n$ converges to $\Gamma$ in $C^\infty(\R/\Z,M)$. Since being a simple curve is a $C^1$-open condition, we conclude that, for $n$ large enough, $\gamma_n$ is a simple curve.
\end{proof}

For the rest of this section, let $M$ be an oriented closed surface, and $L:\Tan M\to\R$ a Tonelli Lagrangian such that $e_0(L)<c_0(L)$. Let $\ggamma=(\gamma_1,...,\gamma_m)$ be an irreducible minimal boundary with energy $c_0(L)$, whose existence is guaranteed by Theorem~\ref{t:below_Mane} and Corollary~\ref{c:irreducible_minimal}. For $i=1,...,m$, we denote by $\mathcal{C}_i$ the connected component of $\W(\R/\Z,M)\times(0,\infty)$ containing $\gamma_i$. If $m>1$, Lemma~\ref{l:irreducible} implies that $\mathcal{C}_i\neq\mathcal{C}_j$ for all $i\neq j$, and none of the $\mathcal{C}_i$'s contains iterated curves.

We recall that the free-period action functional $\SSS_{c_0(L)}$ is bounded from below on every connected component of $\W(\R/\Z,M)\times(0,\infty)$, see \cite[Lemma~4.1]{Abbondandolo:2013is}. Actually,
\begin{align}
\label{e:min_action}
 \SSS_{c_0(L)}(\gamma_i) = \inf_{\mathcal{C}_i} \SSS_{c_0(L)}=:c_i.
\end{align}
Otherwise, there would exists $\gamma_i'\in\mathcal{C}_i$ with $\SSS_e(\gamma_i')<\SSS_e(\gamma_i)$; since $[\gamma_i']=[\gamma_i]$ in $\Hom_1(M;\Z)$, the multicurve $\ggamma'=(\gamma_1,...,\gamma_{i-1},\gamma_i',\gamma_{i+1},...,\gamma_m)$ would be a homological boundary with action 
\[
\SSS_{c_0(L)}(\ggamma')
=\SSS_{c_0(L)}(\ggamma)-\SSS_{c_0(L)}(\gamma_i)+\SSS_{c_0(L)}(\gamma_i')
< \SSS_{c_0(L)}(\ggamma)=0,
\]
contradicting Lemma~\ref{l:actioncrit}.

If $\ggamma$ has $m>1$ components, we define the closed subsets
\begin{align*}
\mathcal{K}_i:= \mathcal{C}_i\cap \SSS_{c_0(L)}^{-1}(c_i),
\qquad i=1,...,m,
\end{align*}
which do not contain iterated curves. If instead $\ggamma$ has only $m=1$ component, we have $c_1=0$, and the intersection $\mathcal{C}_1\cap \SSS_{c_0(L)}^{-1}(c_1)$ also contains iterated curves; in this case, we set
\begin{align*}
\mathcal{K}_1:=\big\{ \zeta\in \mathcal{C}_1\cap \SSS_{c_0(L)}^{-1}(c_1)\ \big|\ \zeta\mbox{ is not iterated}\big\}.
\end{align*}
In both cases, the subsets $\mathcal{K}_i$ are non-empty according to~\eqref{e:min_action}, contain Tonelli waists with energy $c_0(L)$, and do not contain iterated curves.
Consider a multicurve $\zzeta\in(\zeta_1,...,\zeta_m)\in\mathcal{K}_1\times...\times\mathcal{K}_m$. Since every component $\zeta_i$ has the same homology of the corresponding component $\gamma_i$, $\zzeta$ is a homological boundary. Since it has minimal action $\SSS_{c_0(L)}(\zzeta)=\SSS_{c_0(L)}(\ggamma)=0$, its lift is contained in the Mather set $\Mather_0(L)$. Notice that $\zeta_i\neq\zeta_j$ for all $i\neq j$, since $\mathcal{K}_i\cap\mathcal{K}_j=\varnothing$. Moreover, Mather's graph Theorem implies that the components $\zeta_j$'s have pairwise disjoint image and are simple periodic orbits, that is, $\zzeta$ is an embedded homological boundary with action $\SSS_{c_0(L)}(\zzeta)=0$.

\begin{lem}
Each space $\mathcal{K}_i$ is compact.
\end{lem}

\begin{proof}
If $m>1$, we already know that the $\mathcal{K}_i$'s are closed. If $m=1$, we claim that $\mathcal{K}_1$ is closed as well; indeed, on an oriented closed surface, a loop that is the limit of simple loops cannot be an iterated loop; therefore, the closure of $\mathcal{K}_1$ is disjoint from the complement $\mathcal{C}_1\cap \SSS_{c_0(L)}^{-1}(c_1)\setminus\mathcal{K}_1$; since $\mathcal{C}_1\cap \SSS_{c_0(L)}^{-1}(c_1)$ is closed, we infer that $\mathcal{K}_1$ is closed as well\footnote{Actually, if $\gamma_1$ is not contractible, $\mathcal{C}_1$ does not even contain iterated loops, for a non-contractible simple loop  on an oriented closed surface is never freely homotopic to an iterated curve, see e.g. \cite[Proposition~1.4]{Farb:2012rc})}.  
The evaluation map
\begin{align*}
\mathrm{ev}:\mathcal{K}_1\cup...\cup\mathcal{K}_m\to E^{-1}(c_0(L)),
\qquad
\mathrm{ev}(\zeta)=(\zeta(0),\dot\zeta(0))
\end{align*}
is a homeomorphism onto its image. Since the energy level $E^{-1}(c_0(L))$ is compact, we conclude that the closed sets $\mathcal{K}_i$ are compact as well.
\end{proof}

We fix real numbers $p_-,p_+\in(0,\infty)$ such that
\begin{align*}
 p_- & < \min\big\{p\ \big|\ \theta=(\Theta,p)\in\mathcal{K}_1\cup...\cup\mathcal{K}_m\big\},\\
 p_+ & > \max\big\{p\ \big|\ \theta=(\Theta,p)\in\mathcal{K}_1\cup...\cup\mathcal{K}_m\big\}.
\end{align*}
The compactness of $\mathcal{K}_1\cup...\cup\mathcal{K}_m$, together with Lemma~\ref{l:simple_orbits}, implies that each $\mathcal{K}_i$ has an open neighborhood $\mathcal{U}_i\subset \W(\R/\Z,M)\times[p_-,p_+]$ such that all the periodic orbits of the Lagrangian system of $L$ contained in $\mathcal{U}_i$ are simple. Since every multicurve $\zzeta=(\zeta_1,...,\zeta_m)\in\mathcal{K}_1\times...\times\mathcal{K}_m$ is embedded, up to further shrinking the $\mathcal{U}_i$'s, all curves $\zeta\in\mathcal{U}_i$ and $\eta\in\mathcal{U}_j$ have disjoint image provided $i\neq j$.

\begin{lem}\label{l:nbhd}
There exist $\epsilon>0$ and arbitrarily small neighborhoods $\mathcal{W}_i\subseteq \mathcal{U}_i$ of $\mathcal{K}_i$ such that
\begin{align*}
 \inf_{\partial\mathcal{W}_i} \SSS_e > \inf_{\mathcal{W}_i} \SSS_e,
 \qquad
 \forall e\in[c_0(L),c_0(L)+\epsilon).
\end{align*}
\end{lem}

\begin{proof}
We fix $i\in\{1,...,m\}$ once for all. Let $d$ be the distance on $\W(\R/\Z,M)\times(0,\infty)$ induced by its standard Riemannian metric. For each $r>0$, we denote by $\VV_r$ the open neighborhood of $\mathcal{K}_i$ of radius $r$, that is
\begin{align*}
\VV_r
=
\big\{
\gamma\in\W(\R/\Z,M)\times(0,\infty)\ 
\big|\ 
d(\gamma,\zeta)<r\mbox{ for some }\zeta\in\mathcal{K}_i
\big\}.
\end{align*}
We choose $r>0$ small enough so that $\VV_{3r}\subseteq\UU_i$, and the gradient norm $\|\nabla\SSS_{c_0(L)}\|$ is bounded from above on $\VV_{3r}$. We set $\mathcal{K}_i'$ to be the set of critical points of $\SSS_{c_0(L)}$ contained in the closure of $\VV_{3r}\setminus\VV_r$. Since $\SSS_{c_0(L)}$ satisfies the Palais-Smale condition on $\W(\R/\Z,M)\times[p_-,p_+]$, we have that
\begin{align*}
c':=\inf \SSS_{c_0(L)}|_{\mathcal{K}_i'}>c_i.
\end{align*}
We fix $c''\in(c_i,c')$.  The closure of $\{\SSS_{c_0(L)}<c''\} \cap (\VV_{3r}\setminus\VV_r)$ does not contain any critical point of $\SSS_{c_0(L)}$ and, again by the Palais-Smale condition, there exists $\delta>0$ such that $\|\nabla\SSS_{c_0(L)}\|\geq\delta$ on $\{\SSS_{c_0(L)}<c''\} \cap (\VV_{3r}\setminus\VV_r)$.

We set $\WW_i:=\VV_{2r}$, and we claim that 
\[\inf \SSS_{c_0(L)}|_{\partial\WW_i} > \inf \SSS_{c_0(L)}|_{\WW_i}=c_i.\] Indeed, assume by contradiction that there exists a sequence $\{\gamma_j\ |\ j\in\N\}\subset\partial\WW_i$ with $\SSS_{c_0(L)}(\gamma_j)<c_i+1/j$. Fix $0<b<\min\{c''-c_i,\delta r\}$ and $j>1/b$. Notice that the curve $\gamma_j$ is contained in the sublevel set $\{\SSS_{c_0(L)}<c''\}$. Consider the anti-gradient flow $\Phi_s$ of the functional $\SSS_{c_0(L)}$, that is, the partial flow defined by the ordinary differential equation 
\[\tfrac{\diff}{\diff s}\Phi_s=-\nabla\SSS_{c_0(L)}\circ\Phi_s.\] 
Since $\|\nabla\SSS_{c_0(L)}\|$ is bounded from above on $\VV_{3r}$, every orbit $s\mapsto\Phi_s(\gamma)$ that does not stay inside $\VV_{3r}$ for all $s>0$ will eventually hit the boundary of $\VV_{3r}$. Let $s_{\mathrm{out}}\in(0,\infty]$ be the largest real number such that 
\begin{align*}
 \Phi_s(\gamma_j)\in \overline{\VV_{3r}\setminus\VV_r},
 \qquad
 \forall s\in(0,s_{\mathrm{out}}).
\end{align*}
We must have $s_{\mathrm{out}}< b/\delta^2<\infty$, for
\begin{align*}
c_i
< 
\SSS_{c_0(L)}(\Phi_{s_{\mathrm{out}}}(\gamma_j))
=
\SSS_{c_0(L)}(\gamma_j)
-
\int_0^{s_{\mathrm{out}}} \|\nabla\SSS_{c_0(L)}(\Phi_s(\gamma_j))\|^2
<
c_i+b - \delta^2 s_{\mathrm{out}}.
\end{align*}
However, since on the time interval $[0,s_{\mathrm{out}}]$ the curve $s\mapsto\Phi_s(\gamma_j)$ crosses a region of width $r$, we have the estimate
\begin{align*}
\SSS_{c_0(L)}(\Phi_{s_{\mathrm{out}}}(\gamma_j))
& =
\SSS_{c_0(L)}(\gamma_j)
-
\int_0^{s_{\mathrm{out}}} \|\nabla\SSS_{c_0(L)}(\Phi_s(\gamma_j))\|^2 \diff s \\
& < c_i+ b - \delta  \int_0^{s_{\mathrm{out}}} \|\partial_s\Phi_s(\gamma_j)\| \,\diff s \\
& < c_i+b- \delta r\\ 
& < c_i,
\end{align*}
which contradicts the fact that  $\inf \SSS_{c_0(L)}|_{\mathcal{C}_i}=c_i$. This completes the proof of the lemma for $e=c_0(L)$. The general statement follows from this, by observing that \[\SSS_e(\Gamma,p)=\SSS_{c_0(L)}(\Gamma,p) + (e-c_0(L))p. \qedhere\]
\end{proof}

\begin{proof}[Proof of Theorem~\ref{t:just_above}]
Consider the quantity $\epsilon>0$ and the open neighborhoods $\WW_i$ provided by Lemma~\ref{l:nbhd}. Fix $e\in(c_0(L),c_0(L)+\epsilon)$. It is well known that $\SSS_e$ satisfies the Palais-Smale condition on subsets of the form $\W(\R/\Z,M)\times[p_-,p_+]$, see \cite[Lemma~5.3]{Abbondandolo:2013is}. Therefore, for each $i=1,...,m$, there exists a minimizer $\zeta_i$ of $\SSS_e|_{\mathcal{W}_i}$, which is a simple Tonelli waist with energy $e$. By Lemma~\ref{l:boundaries}, the embedded homological boundary $\zzeta=(\zeta_1,...,\zeta_m)$ is a union of finitely many topological boundaries.
\end{proof}

\begin{proof}[Proof of Theorem~\ref{t:multiplicity_just_above}]
Let $M$ be the two-sphere and consider an arbitrary energy level $e\in(c(L),\cw(L))$. Since we are looking for infinitely many periodic orbits with energy $e$, we can assume that the set of critical points $\crit(\SSS_e)$ is a collection of isolated critical circles. Let $\gamma$ be a simple Tonelli waist with energy $e$, which exists by Theorem~\ref{t:just_above}. For each $m\in\N$, we introduce the space of paths
\begin{align*}
\PP_m:=
\big\{
P:[0,1]\toup^{C^0}\W(\R/\Z, S^2)\times(0,\infty)
\ \big|\ 
P(0)\in\gamma,\ P(1)\in\gamma^{m}
\big\},
\end{align*}
and the minimax value
\begin{align*}
s_m
:=
\inf_{P\in \PP_m}
\max_{s\in[0,1]}
\SSS_e(P(s)).
\end{align*}
Since the two-sphere is an oriented surface, every iterate of $\gamma$ is still a local minimizer of $\SSS_e$, see \cite[Lemma~4.1]{Abbondandolo:2015lt}. Therefore \begin{align*}
s_m > \SSS_e(\gamma^m)=m\,\SSS_e(\gamma).
\end{align*}
As $e>c_0(L)$, the free-period action functional $\SSS_e$ is strictly positive and satisfies the Palais-Smale condition at every level on the whole $\W(\R/\Z,S^2)\times(0,\infty)$, see \cite[Lemmas~5.1--5.4]{Abbondandolo:2013is}. Therefore, $s_m$ is a critical value of the free-period action functional $\SSS_e$ and $s_m\to+\infty$ as $m\to+\infty$. Actually, a standard deformation argument from critical point theory allows us to find, for each arbitrarily small neighborhood $\WW\subset \W(\R/\Z,S^2)\times(0,\infty)$ of the set of critical points $\crit(\SSS_e)\cap\SSS_e^{-1}(s_m)$, a path $P\in\PP_m$ such that 
\begin{align}\label{e:almost_optimal_path}
P([0,1])\subset \{\SSS_e<s_m\}\cup\WW.
\end{align}
	
Assume now by contradiction that the Lagrangian system of $L$ admits only finitely many non-iterated periodic orbits $\gamma_1,...,\gamma_k$ with energy $e$. In particular, every such periodic orbit $\gamma_i$ lies on an isolated critical circle $\mathcal Z_i$ of $\SSS_e$. Let us recall the non-mountain pass Theorem for high iterates, which was originally proved in \cite[Theorem 2.6]{Abbondandolo:2014rb} for magnetic Lagrangians, and extended in \cite[Lemma 4.3 and proof of Theorem 1.2]{Asselle:2016qv} to the case of general Tonelli Lagrangians: there exist constants $\overline m_i\in\N$ and, for all $m\geq \overline m_i$, arbitrarily small open neighborhoods $\WW_{i,m}$ of $\mathcal Z_i^m:=\{\zeta^m\ |\ \zeta\in \mathcal Z_i\}$ such that the inclusion induces an injective map between path-connected components
\begin{align}\label{e:non_mountain_pass}
\pi_0(\{\SSS_e<\SSS_e(\gamma_i^m)\})
\hookrightarrow
\pi_0(\{\SSS_e<\SSS_e(\gamma_i^m)\}\cup\WW_{i,m}).
\end{align}
Consider a large enough $m\in\N$ such that $s_m\geq\max\big\{\SSS_e(\gamma_i^{\overline{m}_i})\ \big|\ i=1,...,k\big\}$. The set of critical points $\crit(\SSS_e)\cap\SSS_e^{-1}(s_m)$ is comprised of finitely many critical circles
\begin{align*}
\crit(\SSS_e)\cap\SSS_e^{-1}(s_m)
=
\mathcal Z_{i_1}^{m_1} \cup ... \cup \mathcal Z_{i_k}^{m_k},
\end{align*}
where $m_j\geq\overline m_{i_j}$ for all $j=1,...,k$. We choose the neighborhoods $\WW_{i_j,m_j}$ of $\mathcal Z_{i_j}^{m_j}$ small enough so that they are pairwise disjoint. We set
\begin{align*}
\WW
:=
\WW_{i_1,m_1} \cup ... \cup \WW_{i_k,m_k},
\end{align*}
and we choose a path $P\in\PP_m$ satisfying~\eqref{e:almost_optimal_path}. By~\eqref{e:non_mountain_pass}, we can modify $P$ in order to obtain a new path $Q\in\PP_m$ such that $Q([0,1])\subset\{\SSS_e<s_m\}$. This contradicts the definition of the minimax value $s_m$.
\end{proof}

%%%%%%%%%%%%%%%%%%%%%%%
%%%%%%%%%%%%%%%%%%%%%%%
%%%%%%%%%%%%%%%%%%%%%%%

\section{The non-orientable case}\label{s:non_orientable}

Let $M$ be a closed manifold, and $L$ a Tonelli Lagrangian. For a given  cover $M'\to M$, we denote by $L':\Tan M'\to \R$ the lift of $L$. Clearly $e_0(L')=e_0(L)$. Let $M_0\to M$ and $M_0'\to M'$ be the universal abelian covers, which have fundamental groups 
\[\pi_1(M_0)=[\pi_1(M),\pi_1(M)],
\qquad
\pi_1(M_0')=[\pi_1(M'),\pi_1(M')].\] 
Since $\pi_1(M_0')$ is a subgroup of $\pi_1(M_0)$, $M_0'$ is a cover of $M_0$, and in particular 
\begin{align}\label{e:inequality_c0}
c_0(L')\leq c_0(L).
\end{align}
Moreover, we have the following statement.

\begin{lem}\label{l:c0}
Let $M$ be a closed manifold, $L:\Tan M\to\R$ a Tonelli Lagrangian, and $L':\Tan M'\to\R$ its lift to the tangent bundle of a finite cover $M'$ of $M$. Then
\[c_0(L')= c_0(L).\]
\end{lem}

\begin{proof}
As above, we denote by $M_0$ and $M_0'$ the universal abelian covers of $M$ and $M'$, respectively. We already know that $c_0(L')\leq c_0(L)$.  Let us assume by contradiction that $c_0(L') < c_0(L)$. In particular, there exists a null-homologous periodic curve $\gamma:\R/p\Z\to M$ with action $\SSS_{c_0(L')}(\gamma)<0$. Let $m\in\N$ be the minimal integer such that the homotopy class $[\gamma]^m\in\pi_1(M)$ belongs to the subgroup $\pi_1(M')$. Let $d$ be the number of sheets of the finite cover $M'$. Notice that $m$ divides $d$, and we set $k:=d/m$. The $m$-th iterate of $\gamma$ lifts to periodic curves $\zeta_i:\R/mp\Z\to M'$, for $i=1,...,k$, with pairwise distinct image. Notice that 
\[\SSS_{c_0(L')}'(\zeta_i)=m\SSS_{c_0(L')}(\gamma),\]
where $\SSS_e'$ is the action functional associated to the Lagrangian $L'$. At chain level, the multicurve $\zzeta=(\zeta_1,...,\zeta_k)$ is nothing but the image of $\gamma$ under the transfer map associated to the finite cover $M'$, see \cite[page~321]{Hatcher:2002dt}. Since the transfer map is a chain homomorphism, the multicurve $\zzeta$ is a homological boundary in $M'$. However, 
\begin{align*}
\SSS_{c_0(L')}'(\zzeta)=km\SSS_{c_0(L')}(\gamma)<0,
\end{align*}
which contradicts Lemma~\ref{l:actioncrit}.
\end{proof}
We can now state and prove the main result about periodic orbits on subcritical energy levels on non-orientable surfaces.
\begin{thm}\label{t:non_orientable}
Let $M$ be a closed non-orientable surface, and $L:\Tan M\to\R$ a Tonelli Lagrangian such that $e_0(L)<c_0(L)$. For each energy value $e\in(e_0(L),c_0(L)]$, there exists a homological boundary $\ggamma$ in $M$ with the following properties:
\begin{itemize}
\item[(i)] each component $\gamma_i$ is a simple Tonelli waist of $L$ with energy $e$,
\item[(ii)] for each $i,j\in\{1,...,m\}$, either $\gamma_i$ and $\gamma_j$ coincide  or they have disjoint support,
\item[(iii)] $\ggamma$ has action $\SSS_e(\ggamma)<0$ if $e<c_0(L)$, and action $\SSS_e(\ggamma)=0$ if $e=c_0(L)$,
\item[(iv)] $\ggamma$ lifts to a minimal boundary with energy $e$ for $L'$, where $L'$ is the lift of $L$ to the tangent bundle of the orientation double cover $M'$ of $M$.
\end{itemize}
\end{thm}

\begin{proof}
Let $\pi:M'\to M$ be the orientation double cover of the non-orientable closed surface $M$. We lift $L$ to a Tonelli Lagrangian $L':\Tan M'\to\R$, $L'(q,v)=L(\pi(q),\diff\pi(q)v)$, and we denote by $\SSS_e'$ the action functional associated to $L'$. Clearly $e_0(L)=e_0(L')$ and, by Lemma~\ref{l:c0}, $c_0(L)=c_0(L')$. Fix an energy value $e\in(e_0(L),c_0(L)]$. By Theorem~\ref{t:below_Mane}, there exists a minimal boundary $\ggamma'=(\gamma'_1,...,\gamma'_m)$ with energy $e$ for $L'$ whose action is either $\SSS'_e(\ggamma')<0$ if $e<c_0(L)$, or $\SSS'_e(\ggamma')=0$ if $e=c_0(L)$. The projection $\ggamma=(\gamma_1,...,\gamma_m):=\pi(\ggamma')$ is a homological boundary that satisfies points (iii-iv) of the statement, and whose components are Tonelli waists with energy $e$. Let $F:M'\to M'$ be the non-trivial deck transformation of the orientation double cover. Notice that $\gamma_i(r)=\gamma_j(s)$ for some $(i,r)\neq(j,s)$ if and only if $F(\gamma'_i(r))=\gamma'_j(s)$. Since the Lagrangian $L'$ satisfies $L'(F(q),\diff F(q)v)=L'(q,v)$ for all $(q,v)\in\Tan M'$, the multicurve $F(\ggamma')=(F(\gamma'_1),...,F(\gamma'_m))$ is a minimal boundary. Therefore, by Theorem~\ref{t:graph}, $F(\gamma'_i(r))=\gamma'_j(s)$ if and only if $F(\gamma'_i(r+t))=\gamma'_j(s+t)$ for all $t\in\R$. This completes the proof of point (i), showing that the components of $\ggamma$ are embedded, and of point~(ii).
\end{proof}

%%%%%%%%%%%%%%%%%%%%
%%%%%%%%%%%%%%%%%%%%
%%%%%%%%%%%%%%%%%%%%

\section{Genericity of the condition $e_0(L)<\cu(L)$}
\label{s:suff_cond}
Let $M$ be a closed manifold, and $L:\Tan M\to\R$ a Tonelli Lagrangian with energy $E:\Tan M\to\R$. We denote by $\SSS_e:\W(\R/\Z,M)\times(0,\infty)\to\R\cup\{\infty\}$ the free-period action functional at energy $e$. It is well known that $e_0(L)\leq \cu(L)$. Indeed,  $\SSS_{e_0(L)}(\Gamma,p)=0$ if $\Gamma$ is a constant curve at some $q\in M$ with $e_0(L)=E(q,0)=-L(q,0)$. We set 
\begin{align*}
V(q) & :=-L(q,0),\\
\theta_q(v) & :=\partial_vL(q,0)v,\\
g_q(v,v) & :=\partial^2_{vv} L(q,0)[v,v].
\end{align*}
Notice that $g$ is a Riemannian metric on $M$, since $L$ is Tonelli. Moreover,
\begin{align*}
 e_0(L) = -\min L(\,\cdot\,,0)=\max V.
\end{align*}
We consider the function
\begin{equation*}
\lambda:V^{-1}(e_0(L))\to\R,\qquad \lambda(q)=2 |\hess V(q)|^{1/2}-|\diff\theta_q|,
\end{equation*}
where $|\cdot|$ is the norm induced by $g$, and $\hess V(q)$ is the  Hessian of $V$ at $q$. Notice that this Hessian is well defined, for $V^{-1}(e_0(L))\subset\crit(V)$.

\begin{prop}\label{p:general_suff_cond}
If the function $\lambda$ is somewhere negative, then $e_0(L)<\cu(L)$.
\end{prop}

\begin{proof}
Assume that there exists a point $q\in V^{-1}(e_0(L))$ such that $\lambda(q)<0$, and fix two normal tangent vectors $u,v\in\Tan_qM$ such that $|u|=|v|=1$ and
\begin{equation}\label{e:max}
\diff\theta_{q}(u,v)=|\diff\theta_{q}|\neq0.
\end{equation}
Let $D\subset M$ be an embedded open 2-disk containing the point $q$ and such that $T_qD=\mathrm{span}\{u,v\}$. We orient $D$ so that $v,u$ is an oriented basis of $T_qD$, and we denote by $\mu$ the Riemannian volume form on $D$ induced by $g$ for this orientation. Notice that $\diff\theta|_D=f\mu$ for some smooth function $f:D\to\R$ such that $f(q)< 0$. We set $b:=|\hess V(q)|$. Since $\lambda(q)<0$, we can also fix $a>0$ small enough so that 
\begin{equation}\label{e:neg}
2\sqrt{a+b}+f(q)<0.
\end{equation}
For each $r>0$ smaller than the injectivity radius of $g|_D$ at $q$, we set 
\[e_r:=e_0(L)+\tfrac a2r^2,\] and we denote by $\gamma_r$ the boundary of the Riemannian ball $B_r\subset D$ of $g|_D$ centered at $q$ of radius $r$. We parametrize $\gamma_r$ counterclockwise with constant speed $|\dot\gamma_r(t)|\equiv s_r$ and period $\tau_r $ given by
\[
\tau_r:=\frac{\ell_r}{r\sqrt{a+b}},\qquad s_r:=r\sqrt{a+b},
\]
where $\ell_r$ is the length of $\gamma_r$. Since $\ell_r=2\pi r+o(r)$, we have
\[
\tau_r=\frac{2\pi}{\sqrt{a+b}}+o(1).
\]
Since $e_0(L)=V(q)$, we can estimate
\[
L(x,u) \leq -e_0(L) +\theta_x(u)+ \tfrac{1}{2}|u|^2+\tfrac{1}{2}b\,\mathrm{dist}(q,x)^2 + o( \mathrm{dist}(q,x)^2) + o(|u|^2 ),
\]
where we denoted by ``$\mathrm{dist}$'' the Riemannian distance in $(D,g)$. Therefore
\begin{align*}
L(\gamma_r,\dot\gamma_r)+e_r&=\frac{1}{2}s_r^2+\theta_{\gamma_r}(\dot\gamma_r)+e_r-e_0+\frac{b}{2}r^2+o(r^2)\\
&=\frac{a+b}{2}r^2+\theta_{\gamma_r}(\dot\gamma_r)+\frac{a+b}{2}r^2+o(r^2)\\
&=(a+b)r^2+\theta_{\gamma_r}(\dot\gamma_r)+o(r^2).
\end{align*}
Moreover
\[
\int_0^{\tau_r}\theta_{\gamma_r(t)}(\dot\gamma_r(t))\diff t=\int_{B_r}\diff \theta=\int_{B_r}f\mu=f(q)\pi r^2+o(r^2).
\]
Putting together the last two equations, we obtain the estimate
\begin{align*}
\mathcal S_{e_r}(\gamma_r)&=\tau_r(a+b)r^2+o(r^2)+f(q)\pi r^2+o(r^2)\\
&=\big(2\sqrt{a+b}+f(q)\big)\pi r^2+o(r^2).
\end{align*}
This, together with \eqref{e:neg}, shows that $\mathcal S_{e_r}(\gamma_r)$ is negative for $r>0$ small enough and hence $e_0(L)<e_r<\cu(L)$.
\end{proof}

\newcommand{\TT}{\mathcal{T}}

We denote by $\TT$ the set of Tonelli Lagrangians $L:TM\rightarrow \R$, and by $\TT'\subset\TT$ the subset of those Tonelli Lagrangians $L$ such that $e_0(L)<\cu(L)$.

\begin{prop}\label{p:T'}
The subset $\TT'$ is $C^0$-open and $C^1$-dense in $\mathcal T$.
\end{prop}

\begin{proof}
Consider a Tonelli Lagrangian $L\in\TT'$, and fix an energy $e\in(e_0(L),\cu(L))$ and a curve $\gamma=(\Gamma,p)\in\W(\R/\Z,M)\times(0,\infty)$ such that $\SSS_e(\gamma)<0$. Let $L'\in\TT$ be an arbitrary Tonelli Lagrangian such that $|L-L'|\leq\delta$ on the support of $(\gamma,\dot{\gamma})$ and on the zero section, where
\begin{align*}
\delta = \min \Big\{ \tfrac1p \SSS_e(\gamma) , e-e_0(L)\Big\}
\end{align*}
If $\SSS_e'$ denotes the free-period action functional associated to $L'$, we have 
\[\SSS_e'(\gamma)\leq\SSS_e(\gamma)+p\delta<0,\]
and therefore $e<\cu(L')$. Moreover,
\begin{align*}
e_0(L')=-\min L'(\,\cdot\,,0)\leq -\min L(\,\cdot\,,0) + \delta = e_0(L)+\delta \leq e < \cu(L').
\end{align*}
This proves that $\TT'$ is $C^0$-open in $\mathcal T$.

Now, let $L\in\TT$ be an arbitrary Tonelli Lagrangian and fix an arbitrary $\delta>0$. Let us adopt the notation of Proposition~\ref{p:general_suff_cond}, and consider $V$, $\theta$, $g$, and $\lambda$ associated to $L$. We fix a global maximum $q$ of $V$, so that $V(q)=e_0(L)$ and prove that there exists $L''\in\TT'$ with $|L-L''|_{C^1}<\delta$ in two steps.

First, we claim that there exists $L'\in \TT$ such that $|L-L''|_{C^1}<\delta/2$ and with the property that $V'(q)=e_0(L')$ and $\diff^2 V'(q)=0$. To this purpose, for every arbitrary $\epsilon>0$, we consider a small open ball $B_{r(\epsilon)}(q)$ of radius $r(\epsilon)$ such that $|\diff V|\leq \epsilon$ on $B_{r(\epsilon)}(q)$. There exist a compactly supported function $\chi_\epsilon:B_{r(\epsilon)}(q)\to [0,1]$ which is equal to $1$ on $B_{r(\epsilon)/2}(q)$ and a constant $C$ depending only on a choice of a metric on $M$ such that
\[
|e_0(L)-V|\leq C r\epsilon,\qquad |\diff\chi_\epsilon|\leq C r^{-1},\qquad \text{on } B_{r(\epsilon)}(q).
\]
We define $L'_\epsilon:\Tan M\to \R$ as $L'_\epsilon=L-(e_0(L)-V)\chi_\epsilon$. There holds
\begin{align*}
|L'_\epsilon-L|+|\diff(L'_\epsilon-L)|&=|(e_0(L)-V)\chi_\epsilon|+|\chi_\epsilon \diff V|+|(e_0(L)-V)\diff\chi_\epsilon|\\
&\leq Cr\epsilon+\epsilon+C r\epsilon\cdot Cr^{-1},
\end{align*}
from which we see that $L'_\epsilon\in\TT$ and $|L'_\epsilon-L|_{C^1}<\delta/2$ provided $\epsilon$ is small enough. Moreover, $V'_\epsilon=V+(e_0(L)-V)\chi_\epsilon\leq e_0(L)$ and
\[
V'_\epsilon(q')=V(q')+(e_0(L)-V(q')\chi_\epsilon(q')=e_0(L),\qquad \forall\,q'\in B_{r(\epsilon)/2}(q).
\]
Therefore, we conclude that $e_0(L'_\epsilon)=V'(q)=e_0(L)$ and that $\diff^2 V'(q)=0$, so that we can take $L'=L'_\epsilon$ for $\epsilon$ small.

The second step consists in finding $L''\in \TT'$ with $|L''-L'|_{C^1}\leq \delta/2$. We look for $L''$ of the form
\[
L''_\epsilon(x,v)=L'(x,v)+\epsilon \chi\big(|v|^2\big)\nu_x(v)
\]
for $\epsilon>0$ small. Here, $\chi:[0,\infty)\to[0,1]$ is a smooth bump function supported in $[0,1]$ equal to $1$ in a neighbourhood of $1$ while $\nu$ is some $1$-form on $M$. Clearly, $L''_\epsilon$ converges to $L'$ as $\epsilon$ tends to $0$ and coincides with $L'+\epsilon\nu$ close to the zero section. Therefore, $V''_\epsilon=V'$ and $\theta''_\epsilon=\theta'+\epsilon\nu$. By Proposition~\ref{p:general_suff_cond} it is enough to choose $\nu$ in such a way that the number
\[
\lambda''_\epsilon(q)=2|\diff^2 V''_\epsilon(q)|^{1/2}-|\diff(\theta''_\epsilon)_q|=-|\diff\theta_q+\epsilon\diff\nu_q|
\]
is negative for some small $\epsilon$. This can clearly be achieved and finishes the proof that $\TT'$ is $C^1$ dense in~$\TT$.
\end{proof}

%%%%%%%%%%%%%%%%%%%
%%%%%%%%%%%%%%%%%%%
%%%%%%%%%%%%%%%%%%%

\section{Applications to Finsler geodesic flows of Randers type on $S^2$}
\label{s:Finsler}

Let $L:\Tan S^2\to\R$ be a Tonelli Lagrangian, with associated energy function $E:\Tan S^2\to\R$ and free-period action functionals $\SSS_e$. It is well known that for every $e>c(L)$, the Euler-Lagrange flow of $L$ on the energy hypersurface $E^{-1}(e)$ is orbitally equivalent to the geodesic flow of a Finsler metric on the unit tangent bundle of $S^2$, see \cite[Cor.~2]{Contreras:1998lr}. If the Lagrangian is magnetic, this equivalence is particularly explicit, as we now recall following \cite{Paternain:1999lk}. 

Let $g$ be a Riemannian metric and $\sigma$ an exact 2-form on $S^2$. For every primitive $\theta$ of $L$, we define the magnetic Tonelli Lagrangian
\begin{align*}
L(q,v)=\tfrac12g_q(v,v)+\theta_q(v).
\end{align*}
The Euler-Lagrange flow $\phi_L^t$ and the free-period action functional $\SSS_e$ are independent of the choice $\theta$. In particular, the same is true for the Ma\~n\'e critical value $c(L)$ and for the energy value $\cw(L)$ provided by Theorem~\ref{t:just_above}. The energy function associated to $L$ is given by $E(q,v)=\tfrac12|v|^2$, and thus $e_0(L)=0$. The Tonelli Hamiltonian $H:\Tan^*S^2\to\R$ dual to $L$ is given by
\begin{align*}
 H(q,p) = \tfrac12 |p-\theta_q|^2.
\end{align*}
We denote by $|\cdot|$ the norm of tangent and cotangent vectors associated to the Riemannian metric $g$ on $S^2$, and by $\|\cdot\|_\infty$  the corresponding $L^\infty$-norm of 1-forms, i.e.
\[\|\theta\|_{\infty}:=\max_{q\in S^2} |\theta_q|.\]
By \cite[Theorem~A]{Contreras:1998lr}, we have
\begin{align*}
c(L)= \inf_{\diff \theta=\sigma} \tfrac12 \|\theta'\|_{\infty}^2, 
\end{align*}
where the infimum is taken over all primitives $\theta'$ of $\sigma$, and Proposition~\ref{p:general_suff_cond} implies that $c(L)>e_0(L)$ provided $\sigma$ is not identically zero. We fix $r>\|\theta\|_{\infty}$, and consider the Finsler metric of Randers type on $S^2$
\begin{align*}
F(q,v) = |v| + r^{-1}\theta_q(v).
\end{align*} 
The diffeomorphism
\begin{align*}
\psi: E^{-1}(r^2/2) \to F^{-1}(1),\qquad
\psi(q,v)=(q,F(q,v)^{-1}v)
\end{align*}
realizes an orbit equivalence between the Euler-Lagrange flow of $L$ and the geodesic flow of $F$ on the respective energy hypersurfaces (see \cite[Lemma~2.1]{Paternain:1999lk}). 

We define
\begin{equation}
\label{e:rtheta}
\begin{split}
r_0(g,\sigma) & := \sqrt{ 2\,c(L) }
= \inf_{\diff\theta'=\sigma} \|\theta'\|_{\infty},\\
r_{\mathrm{w}}(g,\sigma) & := \sqrt{ 2\,\cw(L) },
\end{split}
\end{equation}
where the infimum is taken over all primitives $\theta'$ of $\sigma$. Notice that
\begin{align*}
 r_0(g,\sigma) < r_{\mathrm{w}}(g,\sigma),
\end{align*}
according to Theorem~\ref{t:just_above}.
A periodic curve $\Gamma:\R/\Z\to S^2$ is called a \textbf{waist} of the Finsler metric $F$ when it is a local minimizer of the length function 
\begin{align*}
\mathcal{L}: C^\infty(\R/\Z,S^2)\to\R,
\qquad
\mathcal{L}(\Gamma)=\int_{0}^{1} F(\Gamma(t),\dot\Gamma(t))\,\diff t.
\end{align*}
The reparametrization of a waist with constant $F$-speed is a closed geodesic for the Finsler metric $F$. Notice that, if $\gamma=(\Gamma,p)$ is a smooth periodic curve with constant energy $E(\gamma(t),\dot\gamma(t))\equiv \tfrac12 r^2$, then $\SSS_{r^2/2}(\gamma)=\mathcal L(\Gamma)$. This, together with the fact that $\mathcal{L}(\Gamma)$ is independent of the parametrization of $\Gamma$, implies that $\Gamma$ is a waist of $F$ whenever $\gamma=(\Gamma,p)$ is a local minimizer of the free-period action functional $\SSS_{r^2/2}$. Therefore, by  applying Theorems~\ref{t:just_above} and~\ref{t:multiplicity_just_above} to magnetic Lagrangians on $S^2$, we obtain a class of Finsler metrics of Randers type on $S^2$ possessing infinitely many closed geodesics. 

\begin{thm}\label{t:Randers}
Let $g$ be a Riemannian metric and $\sigma$ an exact 2-form on $S^2$ that is not identically zero. Let $\theta$ be any primitive of $\sigma$ such that $\|\theta\|_{\infty}<r_{\mathrm{w}}(g,\sigma)$, and $r$ any positive real number such that  $\|\theta\|_\infty < r < r_{\mathrm{w}}(g,\sigma)$. Then, the Finsler metric \[F(q,v)=g_q(v,v)^{1/2}+r^{-1}\theta_q(v)\] has a simple waist and infinitely many closed geodesics.
\hfill\qed
\end{thm}

In view of Theorem~\ref{t:Randers}, it is useful to have a criterion that guarantees whether a given primitive $\theta$ of $\sigma$ satisfies $\|\theta\|_{\infty}<r_{\mathrm{w}}(g,\sigma)$. One such criterion is provided by the following lemma, which applies in particular when $g$ and $\theta$ are rotationally symmetric.

\begin{lem}\label{l:crit}
Let $M$ be a closed oriented surface, $g$ a Riemannian metric on $M$, and $\theta$ a non-closed 1-form on $M$. Consider
the vector field $Z$ on $M$ defined by $g(Z,\,\cdot\,)=-\theta$, and the set
\[N:=\big\{q\in M\ \big|\ |\theta_q|=\|\theta\|_{\infty}\big\}.\] 
If $N$ contains a topological boundary $\ggamma$ whose components are periodic orbits of the flow of $Z$, then 
$\|\theta\|_{\infty}=r_0(g,\diff\theta)<r_{\mathrm{w}}(g,\diff\theta)$.
\end{lem}

\begin{proof}
We introduce the associated Tonelli Lagrangian $L:\Tan M\to\R$, $L(q,v)=\tfrac12g_q(v,v)+\theta_q(v)$.
All we have to show is that $c_0(L)\geq\tfrac 12 \|\theta\|_\infty^2.$
To this purpose we notice that
\begin{align*}
L(q,v) + \tfrac12\|\theta\|_{\infty}^2
=
\tfrac12 \big( |v+Z(q)|^2 + \|\theta\|_{\infty}^2 - |\theta_q|^2\big).
\end{align*}
In particular
\begin{align}\label{e:Lagrangian_on_N}
L(q,v) + \tfrac12\|\theta\|_{\infty}^2
=
\tfrac12 |v+Z(q)|^2,
\qquad\forall q\in N,\ v\in\Tan_q S^2.
\end{align}
Let $\ggamma=(\gamma_1,...,\gamma_m)\subset N$ be a topological boundary whose components are periodic orbits of the flow of $Z$, and consider the multicurve 
$\zzeta:=\overline\ggamma$ given by reversing the orientation of any component of $\ggamma$. Since every component $\zeta_i$ of $\zzeta$ satisfies
 $\dot\zeta_i(t)=-Z(\zeta_i(t))$, by~\eqref{e:Lagrangian_on_N} we have
\begin{align*}
\SSS_{\|\theta\|_{\infty}^2/2}(\zzeta )
= \sum_{i=1}^m \SSS_{\|\theta\|_{\infty}^2/2}(\zeta_i) = \sum_{i=1}^m
\int_0^{p_i} \tfrac12 |\dot\zeta_i(t)+Z(\zeta_i(t))|^2\diff t=0.
\end{align*}
This implies the desired inequality.
\end{proof}

%%%%%%%%%%%%%%%%%%%%%%
%%%%%%%%%%%%%%%%%%%%%%
%%%%%%%%%%%%%%%%%%%%%%

\appendix

\section{Waists and modifications of the Lagrangian}\label{a:app}

Let $M$ be a closed manifold, and $L:\Tan M\to\R$ a Tonelli Lagrangian with associated energy $E:\Tan M\to\R$, free-period action functional $\SSS_e:\W(\R/\Z,M)\times(0,\infty)\to\R$, and Euler-Lagrange flow $\phi_L^t:\Tan M\to\Tan M$. 
Any Tonelli waist $\gamma$ with energy $e$ is a periodic orbits of the Lagrangian system of $L$ with energy $e$. If we modify the Lagrangian away from $E^{-1}(e)$, $\gamma$ remains a periodic orbit of the new Lagrangian system, but a priori it may not be a minimizer of the new free-period action functional. However, we have the following statement.

\begin{lem}
\label{l:remark_local_minimizers}
Let $L':\Tan M\to\R$ be a Tonelli Lagrangian that coincide with $L$ on the sublevel set $E^{-1}(-\infty,e']$, for some $e'>e>e_0(L)$, and let $\SSS_e'$ be the free-period action functional of $L'$. A local minimizer of $\SSS_e'$ is a local minimizer of $\SSS_e$ as well.
\end{lem}

\begin{rem}
The lemma becomes trivial if one considers the free-period action functionals defined on the space of $C^1$ periodic curves. However, in order to apply global methods from nonlinear analysis, it is more suitable to work with the Sobolev loop space $W^{1,2}(\R/\Z,M)$.
\hfill\qed
\end{rem}

\begin{proof}[Proof of Lemma~\ref{l:remark_local_minimizers}]
Let $\pi:\Tan M\to M$ denote the base projection. As we already recalled in the proof of Lemma~\ref{l:simple_orbits}, there exists $\tau=\tau(e')>0$ small enough such that, for all $q\in M$ and $t\in(0,\tau)$, the map $\pi\circ\phi_L^t|_{\Tan_q M\cap E^{-1}(-\infty,e')}$ is a diffeomorphism onto an open neighborhood of $q$, and for all $\sigma\in(0,\tau)$ and $v\in\Tan_qM$ with $E(q,v)<e'$, the curve 
\begin{align*}
 \gamma:[0,\sigma]\to M,\qquad\gamma(t)=\pi\circ\phi_L^t(q,v)
\end{align*}
is a strict action minimizer: if $\zeta:[0,\sigma]\to M$ is any other absolutely continuous curve such that $\zeta(0)=\gamma(0)$ and $\zeta(\sigma)=\gamma(\sigma)$, then 
\begin{align*}
 \int_0^\sigma L(\gamma(t),\dot\gamma(t))\,\diff t
 <
 \int_0^\sigma L(\zeta(t),\dot\zeta(t))\,\diff t.
\end{align*}

Assume by contradiction that there exists a local minimizer $\gamma=(\Gamma,p)$ of $\SSS_e'$ that is not a local minimizer of $\SSS_e$. Therefore, there exists a sequence $\zeta_n=(Z_n,p_n)\in \W(\R/\Z,M)\times(0,\infty)$ such that $\SSS_e(\zeta_n)<\SSS_e(\gamma)$ and $\zeta_n\to\gamma$ in $\W(\R/\Z,M)\times(0,\infty)$.  Fix $k\in\N$ large enough so that $p/k<\tau$. For all $n$ large enough, we have that $p_n/k<\tau$ as well and, for each $t\in\R/p_n\Z$, the map 
\begin{align*}
\psi_{n,t}:\Tan_{\zeta_n(t)} M\cap E^{-1}(-\infty,e')\to M,
\qquad
\psi_{n,t}(v)=\pi\circ\phi_L^{p_n/k}(\zeta_n(t),v)
\end{align*}
is a diffeomorphism onto a neighborhood of $\{\zeta_n(t),\zeta_n(t+p_n/k)\}$. We set 
\[
v_{n,i}:=\psi_{n,p_ni/k}^{-1}\big(\zeta_n\big(\tfrac{(i+1)p_n}{k}\big)\big),
\qquad
i=0,...,k-1,
\]
and we define the periodic curve $\gamma_n=(\Gamma_n,p_n)\in \W(\R/\Z,M)\times(0,\infty)$ by
\begin{align*}
\gamma_n(\tfrac {i p_n}k+t)
:=
\pi\circ\phi_L^{t}\big(\zeta_n\big(\tfrac {i p_n}k\big), v_{n,i}\big).
\end{align*}
Notice that $\gamma_n$ is a continuous and piecewise broken periodic orbit of the Lagrangian system of $L$ satisfying $\SSS_k(\gamma_n)<\SSS_k(\zeta_n)$. Moreover, since $\gamma_n(\tfrac {i p_n}k)=\zeta_n(\tfrac {i p_n}k)\to \gamma(\tfrac {i p}k)$ as $n\to\infty$, we have that $\Gamma_n|_{[i/k,(i+1)/k]}\to\Gamma|_{[i/k,(i+1)/k]}$ in the $C^\infty$ topology. In particular, for all $n\in\N$ large enough and $t\in\R/p_n\Z$, we have that $E(\gamma_n(t),\dot\gamma_n(t^+))<e'$, which implies 
\begin{align*}
\SSS_e(\gamma_n)
=
\SSS_e'(\gamma_n)
<
\SSS_e'(\gamma)
=
\SSS_e(\gamma).
\end{align*}
This contradicts the fact that $\gamma$ is a minimizer of $\SSS_e$.
\end{proof}

\bibliography{_biblio}
\bibliographystyle{amsalpha}

\end{document}